\documentclass[a4paper,11pt,reqno]{amsart}

\usepackage[verbose]{geometry}
\usepackage[english]{babel}
\usepackage[utf8]{inputenc}

\usepackage{mathtools}
\usepackage{amsmath, amssymb, amsthm}
\usepackage{enumerate}
\usepackage{esint}              %

\usepackage{bbm}                    %
\usepackage{mathbbol}               %
\usepackage[cal=boondoxo]{mathalfa} %
\usepackage{mathrsfs}               %
\usepackage{cases}            %

\usepackage{graphicx}           %
\usepackage{float}              %
\usepackage{tikz}
\usetikzlibrary{patterns}
\usepackage{ifthen}

\usepackage[hang, small, bf]{caption}

\usepackage{hyperref}
\hypersetup{
  pdftitle={A sparse quadratic $T1$ theorem},
  pdfauthor={Gianmarco Brocchi},
  colorlinks=true,
  citecolor=[rgb]{0,0,.5},     %
  linkcolor=[rgb]{0,0,0.5},     
  urlcolor=[rgb]{0,0,0.75},     
  pdfstartview=FitH,
  pdfdisplaydoctitle=true,
  pdflang=en-GB,
  breaklinks=true
}
\usepackage[noabbrev,capitalize]{cleveref} %
\crefformat{section}{\S#2#1#3} %
\crefformat{subsection}{\S#2#1#3}
\crefformat{subsubsection}{\S#2#1#3}

\renewcommand{\roman}[1]{%
  \textup{\uppercase\expandafter{\romannumeral#1}}%
}

\renewcommand{\restriction}{\mathord{\upharpoonright}}
\newcommand{\D}{\, \mathrm{d}}
\newcommand{\1}{ \mathbbm{1}} %

\DeclareMathAlphabet\EuScript{U}{eus}{m}{n} %
\providecommand{\testingC}{  C_{\mathsf{T}} } %
\providecommand{\Stop}[1]{ \mathsf{Tree}(#1)} %
\providecommand{\rStop}[1]{ \mathsf{Tree}_r(#1)}
\providecommand{\ch}{ \mathsf{ch}}
\providecommand{\maxStop}[1]{  \mathsf{ch}_r(#1) } %
\providecommand{\mU}[1]{ #1_{r}} %

\theoremstyle{plain}
\newtheorem{teo}{Theorem}[section]  %
\newtheorem{lemma}[teo]{Lemma}
\newtheorem{prop}[teo]{Proposition}

\theoremstyle{definition}
\newtheorem{define}[teo]{Definition}
\theoremstyle{remark}
\newtheorem{remark}[teo]{Remark}

\numberwithin{equation}{section} %
\numberwithin{table}{section}

\usepackage[backend=biber,backref=true,style=alphabetic]{biblatex}
\AtEveryBibitem{%
  \clearlist{language}%
}

\title[A sparse quadratic $T1$ theorem]{A sparse Quadratic $T(1)$ theorem}
\author{Gianmarco Brocchi}
\address{
  School of Mathematics\\
  University of Birmingham\\
  B15 2TT\\
  Birmingham, UK}
\email{G.Brocchi@pgr.bham.ac.uk}
\thanks{This work was supported by the UK Engineering and Physical Sciences Research Council (EPSRC) grant  EP/L016516/1 for the University of Birmingham.} %

\date{\today}

\begin{document}
\subjclass[2010]{42B20, 42B25}
\keywords{Sparse domination, $T(1)$ theorem, Littlewood--Paley square functions} %

\begin{abstract}
  We show that any Littlewood--Paley square function $S$ satisfying a
  minimal local testing condition
  is dominated by a sparse form,
  \begin{equation*}
    \langle (Sf)^2,g \rangle\le C \sum_{I \in \mathscr{S}} \langle \lvert f\rvert\rangle_I^2 \langle \lvert g\rvert\rangle_I \lvert I\rvert .
  \end{equation*}
  This implies strong weighted $L^p$ estimates for all $A_p$ weights
  with sharp dependence on the $A_p$ characteristic.
  The proof uses random dyadic grids, decomposition in the Haar basis,
  and a stopping time argument.  
\end{abstract}

\maketitle

\section{Introduction}

\subsection{Setting}
Let $\{\theta_t\}_{t>0}$ be a family of integral operators
$\theta_t f(x) = \int_{\mathbb{R}^d} k_t(x,y) f(y) \D{y}$ for which
there exists $C>0$ and $\alpha\in(0,1]$ such that for all $x,y,x',y' \in \mathbb{R}^d$ and $t>0$
the kernels $k_t$ satisfy the following size and regularity conditions:
\begin{align} \label{eq:size_condition} \tag{\textsf{C}1}
  \lvert k_t(x,y) \rvert & \le C \frac{t^\alpha}{(t + \lvert x - y \rvert)^{\alpha + d}}  \\
  \label{eq:smooth_condition} \tag{\textsf{C}2}
  \lvert k_t(x,y) - k_t(x',y')\rvert &\le C \frac{\lvert x - x'\rvert^\alpha + \lvert y - y'\rvert^\alpha}{(t + \lvert x - y \rvert)^{\alpha + d}} && \text{ if }  \lvert x-x'\rvert+ \lvert y-y'\rvert < t  .   
\end{align}
Let $S$ be the square function
\begin{equation*}
  Sf(x) \coloneqq \left( \int_0^\infty \lvert \theta_t f(x) \rvert^2 \frac{\D{t}}{t} \right)^{1/2} .
\end{equation*}
By the $T(1)$ theorem of Christ and Journé \cite{MR906525}
it is known that $S$ is bounded on $L^2(\mathbb{R}^d)$ if
$\theta_t$ applied to the constant function $1$ gives rise to a Carleson measure
 $\nu\coloneqq\lvert\theta_t1(x)\rvert^2 \D{t}/t \D{x}$  on the upper half space $\mathbb{R}^{d+1}_+$.
A Carleson measure on $\mathbb{R}^{d+1}_+$ is a measure which acts like a $d$-dimensional measure in the following sense.
Let $Q$ be a cube in $\mathbb{R}^d$ with sides parallel to the coordinate axes.
Denote by $\ell Q$ and $\lvert Q\rvert$ the side length and the Lebesgue measure of $Q$, so that $(\ell Q)^d = \lvert Q\rvert$.
Consider the Carleson box $B_Q \coloneqq Q \times (0,\ell Q)$.
Then $\nu$ is a Carleson measure if
$\nu(B_Q)/\lvert Q\rvert$ is finite for any cube $Q$.

Let $\1_Q$ be the indicator function on $Q$.
It has been shown \cite{HofmannSurvey,MR2664559,LaceyMartikainen}
that  $S$ is bounded in $L^2(\mathbb{R}^d)$ if there exists a constant $\testingC > 0$ such that
for any cube $Q$
the  following \emph{local} testing condition holds
\begin{equation}\label{eq:testing_condition}\tag{\textsf{T}}
  \int_Q \int_0^{\ell Q} \lvert \theta_t \1_Q(x) \rvert^2 \frac{\D{t}}{t}\D{x}  \le \testingC \lvert Q \rvert .
\end{equation}
A standard example for which these conditions hold is $\theta_t f = f * \psi_t $, where $\psi_t(x) = t^{-d} \psi(t^{-1}x)$
and $\psi$ is the mean zero Schwartz function which gives rise to the Littlewood--Paley square function \cite[\S 6.1]{GrafakosClassical}.
In particular, conditions \eqref{eq:size_condition} and \eqref{eq:smooth_condition} are off-diagonal conditions compatible with the scaling
while \eqref{eq:testing_condition} is the cancellation condition $\int \psi = 0$.

The aim of this paper is to show that this Carleson condition \eqref{eq:testing_condition} is enough to obtain something better: a sparse domination.

\subsection{Main result}
A collection of dyadic cubes $\mathscr{S}$ is $\tau$-sparse if 
for any $Q \in \mathscr{S}$ there exists a subset $E_Q \subset Q$ with the property that
$\{ E_Q \}_{Q\in\mathscr{S}}$ are pairwise disjoint and the ratio $\lvert E_Q \rvert/\lvert Q \rvert \ge \tau$ for a fixed $\tau\in(0,1)$.
  
  Our main result is an optimal sparse domination of $S$ under the minimal condition \eqref{eq:testing_condition}.
  \begin{teo}\label{teo:precise_thm}
    If $S$ satisfies the testing condition \eqref{eq:testing_condition}
    then for any pair of compactly supported functions $f,g\in L^\infty(\mathbb{R}^d)$
    there exists a sparse collection $\mathscr{S}$ such that
    \begin{equation*}
      \Big\lvert \int_{\mathbb{R}^d} (S f)^2 g \D{x}\Big\rvert \le C \sum_{Q \in \mathscr{S}} \left( \frac{1}{\lvert Q\rvert} \int_Q \lvert f \rvert \right)^2 \left(\frac{1}{\lvert Q\rvert} \int_Q \lvert g \rvert \right) \lvert Q \rvert 
    \end{equation*}
    where $C = C(\alpha,d)$ is a positive constant independent of $f$ and $g$.
  \end{teo}

  \subsubsection{Sharp weighted inequalities}  
  Under condition \eqref{eq:testing_condition}
  the square function $S$ is bounded on the weighted space $L^p(w)$ for $p\in(1,\infty)$,
  provided that $w$ belongs to the $A_p$ class of weights for which %
  \begin{equation*}%
    [w]_{A_p} \coloneqq \sup_Q \left(\frac{1}{\lvert Q\rvert} \int_Q w\right) \left(\frac{1}{\lvert Q\rvert} \int_Q w^{-\frac{1}{p-1}}\right)^{p-1} < \infty .
  \end{equation*}
  
  For $p\in(1,\infty)$ and $w$ in $A_p$, let $\alpha(p)$ be the best exponent in the inequality
  \begin{equation}\label{eq:sharp_weighted_estimates}    
    \sup_{f\neq 0} \frac{\lVert S f \rVert_{L^p(w)}}{\lVert f\rVert_{L^p(w)}} \le C(S,p) [w]_{A_p}^{\alpha(p)} .
  \end{equation}
  When $p=2$,
  Buckley \cite{MR1124164} showed the upper bound $\alpha(2)\le 3/2$.
  Later Wittwer improved it to $\alpha(2)=1$ and %
  showed that it's sharp
  for the dyadic %
  and the continuous square functions \cite[Theorem 3.1--3.2]{MR1897458}.
  The same result was obtained independently by Hukovic, Treil and Volberg using Bellman functions \cite[Theorem 0.1--0.4]{MR1771755}.  

  Andrei Lerner was the first to prove
  that $\alpha(p) = \max\{\frac12, \frac{1}{p-1}\}$ cannot be improved \cite[Theorem 1.2]{lerner2006weightednorm}
  and to conjecture estimate \eqref{eq:sharp_weighted_estimates} for Littlewood--Paley square functions.
  After improving the best known exponent for $p>2$ \cite[Corollary 1.3]{lerner2008some},
  Lerner proved the estimate
  \begin{equation}\label{eq:L^3est}
    \lVert Sf \rVert_{L^3(w)} \le C [w]_{A_3}^{1/2} \lVert f\rVert_{L^3(w)} 
  \end{equation}
  for Littlewood--Paley square functions
  pointwise controlled by the intrinsic square function \cite[Theorem 1.1]{lerner2011sharpLP}.
  Lerner achieved this by applying the local mean oscillation formula
  to a dyadic variant of the Wilson intrinsic square function \cite{WilsonIntrisicSquare}.
  Then the sharp estimate \eqref{eq:sharp_weighted_estimates} for all $1<p<\infty$ follows
  from \eqref{eq:L^3est} by the sharp extrapolation theorem \cite{MR2140200}, see also \cite[Theorem 7.5.3]{GrafakosClassical}.
  A proof of the sharp bound \eqref{eq:sharp_weighted_estimates} for the dyadic square function
  using local mean oscillation can be found in \cite[Theorem 1.8]{SharpClassical}.
  
  While Lerner's result relies on a pointwise control of the square function $S$,
  our \cref{teo:precise_thm} implies the weighted estimate \eqref{eq:L^3est} by duality, and so 
  the estimate \eqref{eq:sharp_weighted_estimates} in the full range with optimal
  dependence on the $A_p$ characteristic.

  Weak type estimates \cite{LaceyScurry} and mixed $A_p-A_\infty$ estimates \cite{LaceyLi_mixedtype,borderlineweaktype}
  for square functions have also been studied
  using sparse domination.
  \newline
  After the solution of the $A_2$ conjecture by Hytönen \cite{A2}, sparse domination has been used to obtain
  a simpler proof of the $A_2$ theorem \cite{SimpleA2,ElementaryA2} %
  and to deduce weighted estimates for a plethora of different operators including:
  Calderón--Zygmund operators \cite{SparseRough,CondeRey_DyadicShift,LernerOnPoitwise,UniformSparse_DiPlinio},
  bilinear Hilbert transform \cite{SparseBHT},
  variational Carleson operators \cite{DiPlinioGena}, oscillatory and random singular integrals \cite{SparseOscilRand},
  pseudodifferential operators \cite{beltran2017sparse},
  Stein's square function \cite{MR3625166}, and singular Radon transforms \cite{OberlinRadonSparse}.

  The sparse paradigm has already been extended beyond the classical Calderón--Zygmund theory
  to control more general bilinear forms \cite{beyondCZ} and
  to obtain weighted estimates for Bochner--Riesz multipliers \cite{BR_Reguera,BRmultiplier}.

  Another take on sparse domination, which inspired this work,
  is the sparse $T1$ theorem for Calderón--Zygmund operators \cite{LaceyMena},
  where Lacey and Mena obtained a sparse domination under a minimal testing assumption.

    \subsection{Structure of the paper}
    In \cref{sec:preliminari} we introduce shifted random dyadic grids and the associated Haar basis.
    Furthermore we use the classical reduction to good cubes.
    In \cref{sec:Decomposition1} we decompose the operator into off-diagonal and diagonal parts.
    These are split further each one into two terms
    \begin{equation*}
      \langle (Sf)^2 , g\rangle \lesssim \underbrace{\eqref{term1} + \eqref{term2}}_{\text{off-diagonal}} + \underbrace{\eqref{term3a} + \eqref{term3b}}_{\text{diagonal}}.
    \end{equation*}
    The off-diagonal part is bounded by a dyadic form
    using standard techniques in \cref{sec:term1} and off-diagonal estimates in \cref{sec:term2}.
    The dyadic form is dominated by a sparse form in \cref{sec:sparse_domination}. 

    Terms \eqref{term3a} and \eqref{term3b}
    come from  a Calderón--Zygmund decomposition $g=a+b$,
    where $a$ is  the average part
    and $b$ is the bad part of $g$.
    
    In \cref{sec:term3a} we introduce the stopping cubes used to control the diagonal part.
    We reduce \eqref{term3a} to a telescopic sum on stopping cubes plus off-diagonal terms.
    We remark that the stopping family depends only on the functions $f$ and $g$.
    Furthermore, the testing condition \eqref{eq:testing_condition} is used only in this section and only once.

    In \cref{sec:term3b} we deal with \eqref{term3b}.
    We exploit the zero average property of $b$    together with
    the regularity of the kernel \eqref{eq:smooth_condition}
    to restore a setting in which off-diagonal estimates can be applied as in the previous sections, see \cref{subsec:recover_decay}.
    
    In \cref{sec:proofs_for_reduction} we collect some of the proofs postponed to ease the reading. 
    In \cref{apx:old_tricks} we recall some known results about
    conditional expectations and Haar projections which are used in \cref{sec:term3b}.

    \subsection*{Notation}
      For two positive quantities $X$ and $Y$ the notation $X \lesssim Y$ means that there exists a constant $C>0$
      such that $X \le C Y$. The dependence of $C$ on other parameters
      will be indicated by subscripts $X \lesssim_{d,r,\alpha} Y$ when appropriate.

      Given a cube $Q$ in $\mathbb{R}^d$, the quantities $\partial Q$, $\ell Q$ and $\lvert Q \rvert$
      denote, respectively, boundary, size length, and the Lebesgue measure of $Q$.
      We also denote by $3Q$ the (non-dyadic) cube with the same centre of $Q$ and side length $3\ell Q$.

      The average of a function $f$ over a cube $Q$ will be denoted by
      \begin{equation*}
        \langle f\rangle_Q \coloneqq \fint_Q f \coloneqq \frac{1}{\lvert Q\rvert} \int_Q f(y) \D{y}.
      \end{equation*}

      We consider $\mathbb{R}^d$ with the $\ell^\infty$ metric $\lvert x \rvert = \max_{i} \lvert x_i\rvert$.
      The distance between two cubes $P$ and $R$ will be denoted by $\operatorname{d}(P,R)$,
      while
      \begin{equation*}
        D(P,R) \coloneqq  \ell P + \operatorname{d}(P,R) + \ell R
      \end{equation*}
      is the ``long distance'', as defined in \cite[Definition 6.3]{TbNTV}.

    \section{Preliminaries}\label{sec:preliminari}
   
      \subsection{Dyadic cubes}
      The standard dyadic grid $\mathcal{D}$ on $\mathbb{R}^d$
      is a collection of nested cubes organised in generations
      \begin{equation*}
        \mathcal{D}_j \coloneqq \{ 2^{-j}([0,1)^d + m), m \in \mathbb{Z}^d \} .
      \end{equation*}
      Each generation $\mathcal{D}_j$ is a partition of the whole space
      and   $\mathcal{D} = \cup_{j\in\mathbb{Z}} \mathcal{D}_j$.
      Any cube $Q\in\mathcal{D}_j$ has $2^d$ children in $\mathcal{D}_{j+1}$ and one parent in $\mathcal{D}_{j-1}$.
      For $k\in\mathbb{N}$ we denote by $Q^{(k)}$ the $k$-ancestor of $Q$, that is the unique cube $R$ in the same grid $\mathcal{D}$
      such that $R \supset Q$ and $\ell R = 2^k \ell Q$. We also denote by $\ch_k(Q)$ the set of the $k$-grandchildren of $Q$, so that
      if $P\in \ch_k(Q)$ then $P^{(k)} = Q$.

      \subsection{Haar functions}       %
      Given a dyadic system $\mathcal{D}$ on $\mathbb{R}^d$,
      Haar functions are an orthonormal basis of $L^2(\mathbb{R}^d)$ given by
      linear combinations of indicator functions supported on cubes in $\mathcal{D}$.

      On $\mathbb{R}$, for a given interval $I\in\mathcal{D}$
      let $I^-$ and $I^+$ be the left and the right dyadic child of $I$.
      Consider the functions $h_I^0 \coloneqq \lvert I\rvert^{-1/2} \1_{I}$
      and $h_I^1 \coloneqq ( \1_{I^-} - \1_{I^+} ) \lvert I\rvert^{-1/2}$.
      Then $\{ h_I^1 \}_{I\in\mathcal{D}}$ is an orthonormal complete system of $L^2(\mathbb{R})$.
      In higher dimensions, as a cube $I$ is the product of intervals $I_1 \times \dots \times I_d$,
      a non-constant Haar function $h_I^\epsilon $ is the product $h_{I_1}^{\epsilon_1} \times \dots\times h_{I_d}^{\epsilon_d}$, where $\epsilon = (\epsilon_i)_i \in \{0,1\}^d \setminus \{0\}^d$. 
      
      A function $f$ in $L^2$ can be written in the Haar basis:
      \begin{align*}
        f & = \sum_{I\in\mathcal{D}} \sum_{\epsilon \in \{0,1\}^d \setminus \{0\}^d}\langle f,h_I^\epsilon\rangle h_I^\epsilon \\
          & =  \sum_{I\in\mathcal{D}} \sum_{J \in \ch_1(I)} \left(\langle f \rangle_{J} - \langle f \rangle_I \right) \1_{J} \eqqcolon  \sum_{I\in\mathcal{D}} \Delta_If .
      \end{align*}         
      In this paper the sum over $\epsilon$ is not important, so both the superscript and the sum will be omitted and
      $h_I$ will denote a non-constant Haar function.
      Two bounds that will be used are
      \begin{equation}\label{eq:bounds_L1_L_infinity}
        \lVert \Delta_I f\rVert_{L^1} \le \lvert\langle f,h_I\rangle\rvert \lvert I \rvert^{1/2} \le \int_I \lvert f\rvert, \qquad \lVert \Delta_I f\rVert_{L^\infty} \le \lvert\langle f,h_I\rangle\rvert \lvert I \rvert^{-1/2} \le \fint_I \lvert f\rvert .
      \end{equation}

      \subsection{Good and bad cubes}
      A cube is called \emph{good} if it is distant from the boundary of any much larger cube.
      More precisely, we have the following
      \begin{define}[Good cubes]
        Given two parameters $r \in \mathbb{N}$ and $\gamma \in (0,\frac12)$,
        a cube $R \in \mathcal{D}$ is $r$-good if $\operatorname{d}(R,\partial P)>(\ell R)^\gamma (\ell P)^{1-\gamma}$
        for any  $P \in \mathcal{D}$      with $\ell P \ge 2^r \ell R$.
      \end{define}
      A cube which is not good is a bad cube.

      It is useful to fix $\gamma = \alpha/(4\alpha+4d)$. This is just a convenient choice and
       any other  value of $\gamma$  strictly between $0$ and $\alpha/(2\alpha+2d)$ would work as well.

      \subsection{Shifted dyadic cubes}    
      Given a sequence $\omega = \{\omega_i\}_{i\in\mathbb{Z}} \in (\{0,1\}^d)^\mathbb{Z}$ and a cube $R\in\mathcal{D}_j$ of length $2^{-j}$,
      the translation of $R$ by $\omega$ is defined by
      \begin{equation*}
        R \dot{+} \omega \coloneqq R + x_j \qquad \text{ where } \quad x_j \coloneqq \sum_{i > j} \omega_i 2^{-i} .
      \end{equation*}
      For a fixed $\omega$, let $\mathcal{D}^\omega$ be the collection of dyadic cubes in $\mathcal{D}$ translated by $\omega$.
      The standard dyadic grid corresponds to $\mathcal{D}^0$ where $\omega_i = 0$ for all $i\in\mathbb{Z}$.
      Shifted dyadic grids enjoy the same nested properties of the standard grid $\mathcal{D}^0$,
      together with other properties that will be useful later, see \cref{remark:common_ancestor}.
      For more on dyadic grids, we refer the reader to the beautiful survey \cite[\S 3]{pereyra2018dyadic}.    

      \subsection{Random shifts}\label{subsec:randomgrids}
      Let $\mathbb{P}$ be the unique probability measure on $\Omega \coloneqq (\{0,1\}^d)^\mathbb{Z} $ such that
      the coordinate projections are independent and uniformly distributed.
      Fix $R\in\mathcal{D}^0$ with $\ell R = 2^{-j}$ and consider $J\in\mathcal{D}^0$ with $\ell J > \ell R$.
      The translated cube $J\dot{+}\omega$ is 
      \begin{align*}
        J\dot{+}\omega &= J + \sum_{2^{-i} < \ell R} \omega_i 2^{-i} + \sum_{\ell R \le 2^{-i} < \ell J}  \omega_i 2^{-i} ,\\
        R\dot{+}\omega &= R + \sum_{2^{-i} < \ell R} \omega_i 2^{-i} .
      \end{align*}
      The position of $R\dot{+}\omega$ depends on the $i$ such that $2^{-i} < \ell R$
      while the goodness of $R\dot{+}\omega$, since $R$ and $J$ are translated by the same $\omega$,
      depends on the $i$ such that $2^{-i} \ge \ell R$.
      Then position and goodness of a cube are independent random variables, see \cite{A2}.

      Let $\1_{\text{good}}$ be the function on $\mathcal{D}^\omega$ which takes value $0$ on bad cubes and $1$ on good cubes.
      The probability of a cube $R$ to be good is $\pi_{\text{good}} = \mathbb{P}( R \dot{+} \omega \, \text{is good}) = \mathbb{E}_\omega [\1_{\text{good}}(R \dot{+} \omega)]$,
      where $\mathbb{E}_\omega$ is the expectation with respect to $\mathbb{P}$. %
      The probability $\pi_{\text{good}} > 0$ provided to choose $r$ large enough, see \cite[Lemma 2.3]{RepresentationTH}.
      The indicator function $\1_{R\dot{+}\omega}(\,\cdot\,)$ depends only on the position of $R\dot{+}\omega$, so by the independence of
      goodness and position, for any cube $R\in\mathcal{D}^0$ we have
      \begin{equation}\label{eq:independence}
        \mathbb{E}_\omega [\1_{\text{good}}(R\dot{+}\omega) ] \cdot \mathbb{E}_\omega [\1_{R\dot{+}\omega}(\,\cdot\,)] = \mathbb{E}_\omega [\1_{\{R \dot{+} \omega \text{ good}\}}(\,\cdot\,)] .
      \end{equation}

      \subsection{Calderón--Zygmund decomposition on dyadic grandchildren}
      Let $R$ be a dyadic cube. For $r\in \mathbb{N}$ we denote by $R_r$ a $r$-dyadic child of $R$ in $\ch_r(R)$, so that $R_r^{(r)} = R$.
      \begin{prop}[Calderón--Zygmund decomposition on $r$-grandchildren]\label{prop:C-Z_r}
        Let $r\in\mathbb{N}$ and $f$ be a function in $L^1(\mathbb{R}^d)$.
        For any $\lambda >0$ there exists a collection of maximal dyadic cubes $\mathcal{L}$ and
        two functions $a$ and $b$ 
        such that $f = a + b$, with  $\lVert a \rVert_{L^\infty} \le 2^{d(r+1)} \lambda$ and
        \begin{equation*}
          b \coloneqq \sum_{L\in\mathcal{L}} \sum_{L_{r}\in \ch_r(L)} b_{L_{r}}, \quad \text{ where } \quad b_{L_{r}} \coloneqq \Big( f - \langle f\rangle_{L_{r}} \Big) \1_{L_{r}} .
        \end{equation*}
      \end{prop}
      \begin{remark}
        When $r=0$, this is the usual Calderón--Zygmund decomposition of $f$, see \cite[Theorem 5.3.1]{GrafakosClassical}.
      \end{remark}
      \begin{proof}        
        Given $\lambda>0$, let $\mathcal{L}$ be the collection of maximal dyadic cubes $L$ covering the set
        \begin{equation*}
          E \coloneqq \Big\{ x \in \mathbb{R}^d\,:\, \sup_{Q\in\mathcal{D}} \langle\lvert f\rvert\rangle_Q \1_Q(x) > \lambda \Big\} = \bigcup_{L \in \mathcal{L}} L
        \end{equation*}
        so that
         $\langle\lvert f\rvert\rangle_L \in (\lambda,2^d \lambda]$ for each $L \in\mathcal{L}$.
        Let
        \begin{equation*}
          a \coloneqq f \1_{E^\complement} + \sum_{L\in\mathcal{L}} \sum_{L_{r}\in \ch_r(L)} \langle f \rangle_{L_{r}} \1_{L_{r}} , \qquad b\coloneqq f - a.
        \end{equation*}
        The cubes in $\ch_r(L)$ are a partition of $L$. Since the cubes $L$ in $\mathcal{L}$ are disjoint, we have
        \begin{equation*}
          \lVert a \rVert_{L^\infty} \le  \lambda + \sup_{L\in\mathcal{L}} \sup_{L_{r}\in \ch_r(L)} \lvert \langle f \rangle_{L_{r}} \rvert.
        \end{equation*}
        Let $L^{(1)}$ be the dyadic parent of $L$.
        Then the average of $f$ is controlled by
        \begin{equation*}
          \Big\lvert \frac{1}{\lvert L_r\rvert} \int_{L_{r}} f \Big\rvert \le \frac{\lvert L^{(1)} \rvert}{\lvert L_{r} \rvert} \fint_{L^{(1)}} \lvert f \rvert \le 2^{d(r+1)} \lambda .
        \end{equation*}
      \end{proof}

    \section{Decomposition and good reduction}\label{sec:Decomposition1}
    For any fixed $\omega \in\Omega$  the upper half space $\mathbb{R}^{d+1}_+$ can be decomposed
    in the Whitney regions
    \begin{equation*}
      W_R \coloneqq R \times \left[ \frac{\ell R}2, \ell R \right), \quad R \in \mathcal{D}^\omega .
    \end{equation*}
    Thus we can write
    \begin{equation*}
      \langle (Sf)^2,g\rangle = \iint_{\mathbb{R}^{d+1}_+} \lvert \theta_t f(x)\rvert^2 \frac{\D{t}}t g(x) \D{x}
      = \sum_{R\in\mathcal{D}^\omega} \iint_{W_R} \lvert \theta_t f(x)\rvert^2 \frac{\D{t}}t g(x) \D{x} .
    \end{equation*}
    Then we decompose $f = \sum_{P\in\mathcal{D}^\omega} \Delta_P f$. Given $R\in\mathcal{D}^\omega$, we distinguish
    two collections of $P$:
    \begin{equation*}
      \mathcal{P}^\omega_{R} \coloneqq \{P\in\mathcal{D}^\omega \,:\, P\supset R^{(r)} \}, \quad \text{ and } \quad \mathcal{D}^\omega \setminus \mathcal{P}^\omega_{R} .
    \end{equation*}
    We shall sometimes omit the superscript $\omega$ in the following.
    Bound the operator: 
    \begin{align}
      \sum_{R\in\mathcal{D}} \iint_{W_R} \lvert \theta_t f(x)\rvert^2 \frac{\D{t}}t g \D{x} 
      & \le 2 \sum_{R\in\mathcal{D}} \iint_{W_R} \Big( \big\lvert\sum_{P\in\mathcal{D} \setminus \mathcal{P}_{R}} \theta_t \Delta_Pf \big\rvert^2 + \big\lvert \sum_{P\in\mathcal{P}_{R}} \theta_t \Delta_Pf \big\rvert^2 \Big) \lvert g\rvert \frac{\D{t}}t \D{x} \label{eq:1st_splitting}.
    \end{align}
    Consider the second term in \eqref{eq:1st_splitting}.
    Let $P_R$ be the dyadic child of $P$ containing $R$.
    Then $\Delta_P f \1_P = \Delta_P f \1_{P\setminus P_R} + \langle \Delta_P f\rangle_{P_R} \1_{P_R}$
    and we split the operator accordingly as before to obtain:
    \begin{align}
      \sum_{R\in\mathcal{D}} & \iint_{W_R} \lvert \theta_t f(x)\rvert^2 \frac{\D{t}}t g \D{x} \lesssim  \nonumber \\
      & \sum_{R\in\mathcal{D}} \iint_{W_R} \Big\lvert\sum_{P\in\mathcal{D} \setminus \mathcal{P}_{R}} \theta_t \Delta_Pf \Big\rvert^2 \lvert g\rvert \frac{\D{t}}t \D{x} \label{term1} \tag{\roman{1}} \\
      + & \sum_{R\in\mathcal{D}} \iint_{W_R} \Big\lvert\sum_{P\in\mathcal{P}_{R}} \theta_t \Delta_P f \1_{P\setminus P_R} \Big\rvert^2 \lvert g\rvert \frac{\D{t}}t \D{x} \label{term2} \tag{\roman{2}} \\ 
      + & \sum_{R\in\mathcal{D}} \iint_{W_R} \Big\lvert \sum_{P\in\mathcal{P}_{R}} \theta_t \langle \Delta_P f\rangle_{P_R} \1_{P_R} \Big\rvert^2 \lvert g\rvert \frac{\D{t}}t \D{x} \label{term3} \tag{\roman{3}}.
    \end{align}
    In each term, without loss of generality, we can assume $g$ to be supported on $R$.
    We write $\lvert g\rvert = a + b$ using the Calderón--Zygmund decomposition in \cref{prop:C-Z_r} at height $\lambda= A\langle\lvert g\rvert\rangle_R$ for $A>1$.
    Then the bad part $b$ is decomposed in the Haar basis.
    \begin{align}
      \eqref{term3} =
      & \sum_{R\in\mathcal{D}} \iint_{W_R} \Big\lvert\sum_{P\in\mathcal{P}_{R}} \theta_t\langle \Delta_P f\rangle_{P_R} \1_{P_R} \Big\rvert^2 \frac{\D{t}}t a(x) \D{x}  \label{term3a}\tag{$\roman{3}_a$}\\
      & + \sum_{R\in\mathcal{D}} \iint_{W_R} \Big\lvert \sum_{P\in\mathcal{P}_{R}} \theta_t\langle \Delta_P f\rangle_{P_R} \1_{P_R} \Big\rvert^2 \sum_{\substack{Q\in\mathcal{D}\\ Q\subset R}} \Delta_Q b(x) \frac{\D{t}}t \D{x}.       \label{term3b}\tag{$\roman{3}_b$}
    \end{align}

    \subsection{Good reduction}
    Averaging over all dyadic grids $\mathcal{D}^\omega$   we have
    \begin{align*}
      \iint_{\mathbb{R}^{d+1}_+} \lvert\theta_t f\rvert^2 \lvert g\rvert \frac{\D{t}}t \D{x} &= \mathbb{E}_\omega \sum_{R\in\mathcal{D}^\omega} \iint_{W_R}\lvert\theta_t f\rvert^2 \lvert g\rvert \frac{\D{t}}t \D{x} \\    
                                                                                        & \lesssim \mathbb{E}_\omega \big[\roman{1} + \roman{2} + \roman{3} \big]
                                                                                          = \mathbb{E}_\omega \big[\roman{1} + \roman{2} + \roman{3}_a\big] + \mathbb{E}_\omega\big[\roman{3}_b\big] 
    \end{align*}
    because all the integrands are non-negative
    and the expectation $\mathbb{E}_\omega$ is linear.

    By using the identity \eqref{eq:independence} and writing $1$ as $\pi_{\text{good}}^{-1}\mathbb{E}_\omega [\1_{\text{good}}(\,\cdot\, \dot{+} \omega)]$,
    one can turn a sum over all cubes in $\mathcal{D}^\omega$ into a sum over good cubes, in particular:
    \begin{gather}\label{eq:good_cubes_reduction}
      \mathbb{E}_\omega \big[\roman{1} + \roman{2} + \roman{3}_a \big]
      = \pi_{\text{good}}^{-1} \mathbb{E}_\omega \big[ \1_{\text{good}}(R \dot{+} \omega) \big(\roman{1} + \roman{2} + \roman{3}_a \big) \big] ,\\
      \mathbb{E}_\omega \big[\roman{3}_b \big] 
      = \pi_{\text{good}}^{-1} \mathbb{E}_\omega \big[ \1_{\text{good}}(Q \dot{+} \omega) \big(\roman{3}_b \big) \big] \nonumber.
    \end{gather}    
    We refer the reader to \cite[\S 2.2]{Squarefungeneralmeasures} for an expanded version of \eqref{eq:good_cubes_reduction}
    with $g \equiv 1$.
    
    From now on, the cubes $Q$ in \eqref{term3b} and the cubes $R$ in all other cases are considered to be good cubes.
    The superscript in $\mathcal{D}^\omega$, as well as the expectation $\mathbb{E}_\omega$ and the probability $\pi_{\text{good}}$ will be omitted.

    \section{Reduction of \texorpdfstring{\eqref{term1}}{(I)} to a dyadic form}\label{sec:term1}
    We start by showing that
    \begin{equation*}
      \eqref{term1} = \sum_{\substack{R\in\mathcal{D}\\R \text{ good}}} \iint_{W_R} \big\lvert\sum_{P\in\mathcal{D} \setminus \mathcal{P}_{R}} \theta_t \Delta_Pf \big\rvert^2 \lvert g\rvert \frac{\D{t}}t \D{x}
      \lesssim \sum_{j\in\mathbb{N}} 2^{-c j} B_j^{\mathcal{D}}(g,f)
    \end{equation*}
    for $c>0$, where $B_j^{\mathcal{D}}(g,f)$ is the dyadic form given by
    \begin{equation} \label{eq:def_B_j}
      B_j^{\mathcal{D}}(g,f) \coloneqq  \sum_{K \in \mathcal{D}} \langle\lvert g\rvert\rangle_{3K}  \sum_{\substack{P\in\mathcal{D}\\P \subset 3K \\ \ell P = 2^{-j}\ell K}} \langle f , h_P \rangle^2 .
    \end{equation}    
    We remark that the function $g$ barely plays any role in this section.
    
    \subsection{Different cases for \texorpdfstring{$P$}{P}}\label{subsec:different_cases}
    Given $R \in \mathcal{D}$, the cubes $P$ are grouped
    according to their length and position with respect to $R$.
    \begin{table}[H]
      \caption{Different cases for \(P\) given \(R\) according to their lengths (first row) and position.}
      \begin{tabular}{|c|c|c|c|c|c|c|c|}
        \multicolumn{4}{l|}{\hspace{1cm}$\ell P \ge 2^{r+1} \ell R$} & \multicolumn{2}{c|}{$\ell R\le \ell P \le 2^{r} \ell R$} & \multicolumn{2}{c}{$\ell P < \ell R$} \\
        \cline{1-1} \cline{3-8}
        $P \supset R$ & & \multicolumn{2}{|c|}{$P \not\supset R$} & & & \multicolumn{2}{c|}{} \\ 
                                                        & & \multicolumn{2}{c|}{} & & & \multicolumn{2}{c|}{$\mathcal{P}_{\text{subscale}}$} \\            
                                                        & & $3P\setminus P \supset R$ & $3P \not\supset R$ & $3P \not\supset R$ & $3P \supset R$ & $P \subset 3R$ & $P \not\subset 3R$ \\
                                                                    & & $\mathcal{P}_{\text{near}}$ & \multicolumn{2}{c|}{$\mathcal{P}_{\text{far}}$} & $\mathcal{P}_{\text{close}}$ & inside & far \\[-2ex]
                                                                    & &  \multicolumn{6}{c|}{} \\ 
        $\mathcal{P}_{R}$ & & \multicolumn{6}{c|}{$\mathcal{D} \setminus \mathcal{P}_{R}$} \\
        \cline{1-1} \cline{3-8}
      \end{tabular}
      \label{tab:cases}
    \end{table}
    \begin{remark}
      Since $3P$ is the union of $3^d$ cubes in $\mathcal{D}$,
      the condition $3P \not\supset R$ is equivalent to $3P \cap R =\emptyset$, which implies that $\operatorname{d}(P,R) > \ell P$.
      The condition $\ell P \ge 2^{r+1}\ell R$ allows to exploit the goodness of $R$ also with dyadic children of $P$.
    \end{remark}

    We decompose the sum over $P\in\mathcal{D}\setminus\mathcal{P}_{R}$ in four terms.
    \begin{flalign}
      \sum_{R\in\mathcal{D}} \iint_{W_R} \Big\lvert \sum_{P\in\mathcal{D} \setminus \mathcal{P}_{R}} \theta_t (\Delta_P f) \Big\rvert^2 \lvert g\rvert \frac{\D{t}}t \D{x} %
      \lesssim & \sum_{R\in\mathcal{D}} \iint_{W_R} \Big\lvert \sum_{\substack{P\,:\,\ell P> 2^r \ell R\\ 3P\setminus P\supset R}} \theta_t(\Delta_P f) \Big\rvert^2 \lvert g \rvert \frac{\D{t}}{t} \D{x} \tag{near}\label{near}\\
      & + \sum_{R\in\mathcal{D}} \iint_{W_R} \Big\lvert \sum_{\substack{P \,:\, \ell P \ge \ell R \\ \operatorname{d}(P,R)>\ell P}} \theta_t(\Delta_P f) \Big\rvert^2 \lvert g \rvert \frac{\D{t}}{t} \D{x} \tag{far} \label{far}\\
      & + \sum_{R\in\mathcal{D}} \iint_{W_R} \Big\lvert \sum_{\substack{P \,:\, 3P\supset R \\ \ell R\le \ell P \le 2^r \ell R}} \theta_t(\Delta_P f) \Big\rvert^2 \lvert g \rvert \frac{\D{t}}{t} \D{x} \tag{close}\label{close}\\        
      & + \sum_{R\in\mathcal{D}} \iint_{W_R} \Big\lvert \sum_{P \,:\, \ell P < \ell R} \theta_t(\Delta_P f) \Big\rvert^2 \lvert g \rvert \frac{\D{t}}{t} \D{x} \tag{subscale}\label{subscale}.
    \end{flalign}

    \subsection{Estimates case by case}
    We start with a well--known bound.
    \begin{lemma}\label{lemma:LM}
      Let $P,R \in \mathcal{D}$ with $R$ good. If one of the following conditions holds
      \begin{enumerate}
      \item  $\ell P \ge \ell R$ and $P$ and $R$ are disjoint;
      \item $\ell P < \ell R$;
      \end{enumerate}
      then for $(x,t) \in W_R$ we have
      \begin{equation*}%
        \lvert \theta_t (\Delta_Pf)(x)\rvert \lesssim \frac{(\sqrt{\ell R \ell P})^\alpha}{D(R,P)^{\alpha+d}} \lVert \Delta_Pf \rVert_{L^1} .
      \end{equation*}
    \end{lemma}
    
    The proof uses the goodness of $R$ in case $(1)$ and the zero average of $\Delta_Pf$ in case $(2)$,
    see also \cite[\S 5]{LaceyMartikainen},\cite[\S 2.4]{Squarefungeneralmeasures}.
    Details of the proof are deferred to \cref{sec:proofs_for_reduction}.
    \newline
    We apply \cref{lemma:LM} for $P$ in $\mathcal{P}_i$ with $i \in \{$near, far, close, subscale$\}$
    and estimate $\lVert \Delta_P f\rVert_{L^1}$ as in \eqref{eq:bounds_L1_L_infinity}.
    Then we apply Cauchy--Schwarz in $\ell^2$.
    \begin{gather}\label{eq:First_Cauchy_Schwarz}
      \sum_{R \in \mathcal{D}} \iint_{W_R} \Big\lvert \sum_{P\in\mathcal{P}_{i}} \theta_t(\Delta_P f) \Big\rvert^2 \lvert g\rvert \frac{\D{t}}{t} \D{x} 
      \lesssim \sum_{R \in \mathcal{D}} \iint_{W_R} \left( \sum_{P\in\mathcal{P}_{i}} \lvert \langle f , h_P \rangle\rvert \frac{(\sqrt{\ell R \ell P})^\alpha}{D(R,P)^{\alpha+d}} \lvert P \rvert^{1/2} \right)^2 \lvert g\rvert \frac{\D{t}}{t} \D{x} \nonumber \\
      \le \sum_{R \in \mathcal{D}} \iint_{W_R} \left( \sum_{P\in\mathcal{P}_{i}} \langle f , h_P \rangle^2  \frac{(\sqrt{\ell R \ell P})^\alpha}{D(R,P)^{\alpha+d}}
        \cdot \sum_{P\in\mathcal{P}_{i}} \frac{(\sqrt{\ell R \ell P})^\alpha}{D(R,P)^{\alpha+d}} \lvert P \rvert \right) \lvert g\rvert \frac{\D{t}}{t} \D{x}.
    \end{gather}
  The quantity in parenthesis in \eqref{eq:First_Cauchy_Schwarz} does not depend on $t$,
  so we  bound $ \int_{\ell R/2}^{\ell R} \D{t}/t \le 1$ by taking the supremum in $t$.
  The second factor after Cauchy--Schwarz is finite in all cases.
  \begin{lemma}\label{lemma:2nd_factor_finite}
    Let $i \in \{$\ref{near}, \ref{far}, \ref{close}, \ref{subscale}$\}$, then
    \begin{equation*}
      \sum_{P\in\mathcal{P}_i} \frac{(\sqrt{\ell R \ell P})^\alpha}{D(R,P)^{\alpha+d}} \lvert P \rvert \lesssim 1 .
    \end{equation*}
  \end{lemma}
  Details of the proof are in \cref{sec:proofs_for_reduction}.  
    We proceed with studying
    \begin{equation*}
      \sum_{R \in \mathcal{D}} \int_{R} \left( \sum_{P\in\mathcal{P}_i} \langle f , h_P \rangle^2  \frac{(\sqrt{\ell R \ell P})^\alpha}{D(R,P)^{\alpha+d}} \right) \lvert g\rvert \D{x}
    \end{equation*}
    for $i \in \{$near, far, close, subscale$\}$.
    When $P$ and $R$ are disjoint, it's useful to rearrange the sums using
    a common ancestor of $P$ and $R$.

      \begin{lemma}[Common ancestor]\label{lemma:adapted_TH} 
        Let $R, P \in\mathcal{D}$ be disjoint cubes with $R$ good. \\
        If $\operatorname{d}(R,P) > \max(\ell R,\ell P)^{1-\gamma} \min(\ell R, \ell P)^\gamma$
        then there exists $K \supseteq P \cup R$ such that
        \begin{equation*}
          \ell K  \left(\frac{\min(\ell P,\ell R)}{\ell K}\right)^\gamma \le 2^r \operatorname{d}(R,P) .
        \end{equation*}
      \end{lemma}
      A proof in the case $\ell P \ge \ell R$ can be found in \cite[Lemma 3.7]{RepresentationTH}.     
      When $\ell P < \ell R$, the same ideas carry over, see \cref{sec:proofs_for_reduction} for a proof of this case.

      \begin{remark}\label{remark:common_ancestor}
        For any $P,R \in \mathcal{D}^\omega$ there exists (almost surely) a common ancestor $K \in \mathcal{D}^\omega$.
        Indeed, dyadic grids (like the standard grid $\mathcal{D}^0$) without this property
        have zero measure in the probability space $(\Omega,\mathbb{P})$,
        see \cite[\S 3.1.1 and Example 3.2]{pereyra2018dyadic}.
      \end{remark}

      \subsection{\texorpdfstring{$P$}{P} far from \texorpdfstring{$R$}{R}}\label{sec:far}
      In this case $\operatorname{d}(P,R) > \ell P$ and $\ell P = \max(\ell P,\ell R)$, so the hypotheses of \cref{lemma:adapted_TH} are satisfied.
    Let $K$ be the common ancestor of $P$ and $R$
    given by \cref{lemma:adapted_TH}.
    Since $\ell P \ge 2^{r+1}\ell R$, let $\ell P = 2^{-j}\ell K$ and  $\ell R = 2^{-i-j} \ell K$
    for some $i,j \in \mathbb{Z}_+$, with $i\ge r+1$. We have
    \begin{equation*}
      \sum_{R \in \mathcal{D}} \int_{R} g \left( \sum_{P\in\mathcal{P}_\text{far}}\langle f , h_P \rangle^2  \frac{(\sqrt{\ell R \ell P})^\alpha}{\operatorname{d}(R,P)^{\alpha+d}} \right)
      = \sum_{K\in\mathcal{D}} \sum_{i,j} \sum_{\substack{R\,:\,R\subset K\\\ell R=2^{-i-j}\ell K}} \int_R g \sum_{\substack{P\,:\,P\subset K\\\ell P=2^{-j}\ell K\\ \operatorname{d}(P,R)>\ell P}} \langle f , h_P \rangle^2  \frac{(\sqrt{\ell R \ell P})^\alpha}{\operatorname{d}(R,P)^{\alpha+d}}.
    \end{equation*}    
    By using the lower bound $\operatorname{d}(P,R) \gtrsim_r (\ell K)^{1-\gamma} (\ell R)^\gamma$ with $\gamma = \alpha/(4\alpha+4d)$, we estimate
    \begin{equation}\label{eq:bound_the_coefficients}
      \frac{\sqrt{\ell P \ell R}}{\operatorname{d}(P,R)} \lesssim_r \frac{2^{-j - i/2} \ell K}{\ell K 2^{-(i+j)\gamma}} \quad
      \text{ so that } \quad
      \frac{(\sqrt{\ell P \ell R})^\alpha}{\operatorname{d}(P,R)^{\alpha+d}} \lesssim_{r,\alpha,d} \frac{2^{-(j+i/2)\alpha}}{2^{-(i+j)\gamma(\alpha+d)}\lvert K \rvert} =
      \frac{2^{-(3j+i)\alpha/4}}{\lvert K \rvert}.
    \end{equation}
    For any fixed integer $m$, the set $\{ R \subset K \,:\, \ell R = 2^{-m}\ell K\}$ is a partition of $K$, so    we bound 
    \begin{flalign*}
      \sum_{K\in\mathcal{D}} \sum_{i,j} \sum_{\substack{R\,:\,R\subset K\\\ell R=2^{-i-j}\ell K}} \int_R g \sum_{\substack{P\,:\,P\subset K\\\ell P=2^{-j}\ell K\\ \operatorname{d}(P,R)>\ell P}} \langle f , h_P \rangle^2  \frac{(\sqrt{\ell R \ell P})^\alpha}{\operatorname{d}(R,P)^{\alpha+d}} \\
      \lesssim \sum_{j \in \mathbb{N}} 2^{-3j\alpha/4} \sum_{i \ge r+1} 2^{-i\alpha/4} \sum_{K \in \mathcal{D}} \fint_K \lvert g\rvert  \sum_{\substack{P\,:\,P \subset K \\ \ell P = 2^{-j}\ell K}} \langle f , h_P \rangle^2 .
    \end{flalign*}
    We can sum in $i$, then
    \begin{flalign*}
      \sum_{j \in \mathbb{N}} 2^{-3j\alpha/4} \sum_{K \in \mathcal{D}} \langle\lvert g\rvert\rangle_K \sum_{\substack{P \subset K \\ \ell P = 2^{-j}\ell K}} \langle f , h_P \rangle^2
      & \le 3^d \sum_{j \in \mathbb{N}} 2^{-3j\alpha/4} \sum_{K \in \mathcal{D}} \langle\lvert g\rvert\rangle_{3K}  \sum_{\substack{P \subset 3K \\ \ell P = 2^{-j}\ell K}} \langle f , h_P \rangle^2  \\
      & = \sum_{j \in \mathbb{N}} 2^{-3j\alpha/4} B_j^{\mathcal{D}}(g,f). 
    \end{flalign*}
    A sparse domination of $B_j^{\mathcal{D}}(g,f)$ is proved in \cref{sec:sparse_domination}.
    
    \subsection{\texorpdfstring{$P$}{P} near \texorpdfstring{$R$}{R}}
    Recall that 
    $P\in\mathcal{P}_{\text{near}}$ if $3P \setminus P \supset R$ and $\ell P \ge 2^{r+1} \ell R$. By the goodness of $R$, we have that
    $\operatorname{d}(P,R) > (\ell P)^{1-\gamma} (\ell R)^\gamma $. So the hypotheses of \cref{lemma:adapted_TH} are satisfied
    and there exists $K\supseteq P\cup R$ such that $\operatorname{d}(P,R) \gtrsim_r (\ell K)^{1-\gamma} (\ell R)^\gamma$.
    Arguing as in the far term leads to    
    \begin{flalign*}
      \sum_{R \in \mathcal{D}} \int_{R} g \left( \sum_{P\in\mathcal{P}_\text{near}}\langle f , h_P \rangle^2  \frac{(\sqrt{\ell R \ell P})^\alpha}{\operatorname{d}(R,P)^{\alpha+d}} \right)
      \lesssim \sum_{j \in \mathbb{N}} 2^{-3j\alpha/4} B_j^{\mathcal{D}}(g,f) .
    \end{flalign*}
    
    \subsection{\texorpdfstring{$P$}{P} comparable and close to \texorpdfstring{$R$}{R}}
    In this case $\ell R\le \ell P \le \ell R^{(r)}$ and $3P \supset R$. Using the trivial bound $D(P,R) \ge \ell R$
    we have    
    \begin{flalign*}
      \sum_{R \in \mathcal{D}} \int_{R}\lvert g\rvert\left( \sum_{P\in\mathcal{P}_\text{close}}\langle f , h_P \rangle^2  \frac{(\sqrt{\ell R \ell P})^\alpha}{D(R,P)^{\alpha+d}} \right)
      & \lesssim_{r,\alpha} \sum_{R \in \mathcal{D}} \int_R\lvert g\rvert\sum_{\substack{P\,:\,3P\supset R\\ \ell R\le \ell P \le 2^r\ell R}} \langle f, h_P \rangle^2 \frac{1}{\lvert R \rvert}.
    \end{flalign*}
    Rearrange the sum in groups of $P$ such that $\ell P = 2^k \ell R$ for $k\in \{0,\dots,r\}$. Then
    \begin{flalign*}
      \sum_{R \in \mathcal{D}} \int_R\lvert g\rvert\sum_{k = 0}^r \sum_{\substack{P\,:\,3P\supset R\\ \ell P = 2^k\ell R}} \langle f, h_P \rangle^2 \frac{1}{\lvert R \rvert}
      & = \sum_{k = 0}^r \sum_{P\in\mathcal{D}} \langle f, h_P \rangle^2 \frac{2^{kd}}{\lvert P \rvert} \sum_{\substack{R \subset 3P \\ \ell R = 2^{-k}\ell P}} \int_R\lvert g\rvert \\
      & \le \sum_{k = 0}^r \sum_{P\in\mathcal{D}} \langle f, h_P \rangle^2 \frac{2^{kd}}{\lvert P \rvert} \int_{3P} \lvert g\rvert \\
      & \lesssim_{r,d} \sum_{P\in\mathcal{D}} \langle f, h_P \rangle^2 \frac{3^{d}}{\lvert 3P \rvert} \int_{3P} \lvert g\rvert = 3^d \sum_{P\in\mathcal{D}} \langle f, h_P \rangle^2 \langle \lvert g\rvert \rangle_{3P} .
    \end{flalign*}
    We define
    \begin{equation}
      \label{eq:def_B_0}
      B_0^{\mathcal{D}}(g,f) \coloneqq \sum_{P\in\mathcal{D}} \langle f, h_P \rangle^2 \langle \lvert g\rvert \rangle_{3P} .
    \end{equation}
    Then $B_0^{\mathcal{D}}(g,f)$ is bounded by a sparse form in \cref{sec:sparse_domination}.

    \subsection{Subscale}\label{sec:P_small}
    When $\ell P < \ell R$ we distinguish two subcases, as shown in \cref{tab:cases}.
    \subsubsection{Inside \texorpdfstring{: $P \subset 3R$}{}}\label{subsec:subscale_inside}
    The leading term in the long-distance $D(R,P)$ is $\ell R$, so we bound
    \begin{flalign*}
      \sum_{R \in \mathcal{D}} \int_{R}\lvert g\rvert\Bigg( \sum_{\substack{P\,:\,\ell P < \ell R\\P\subset 3R}}\langle f , h_P \rangle^2  \frac{(\sqrt{\ell R \ell P})^\alpha}{D(R,P)^{\alpha+d}} \Bigg)
      & \le  \sum_{R\in\mathcal{D}} \fint_R\lvert g\rvert \sum_{\substack{P\,:\,\ell P < \ell R\\P\subset 3R}} \langle f , h_P \rangle^2 \left(\frac{\ell P}{\ell R}\right)^{\alpha/2} \\
      & = \sum_{j\in\mathbb{N}} 2^{-j\alpha/2} \sum_{R\in\mathcal{D}} \langle \lvert g\rvert\rangle_R  \sum_{\substack{P\,:\,P \subset 3R \\ \ell P = 2^{-j}\ell R}}\langle f,h_P\rangle^2 \\
      & \lesssim_d \sum_{j\in\mathbb{N}} 2^{-j\alpha/2} B_j^{\mathcal{D}}(g,f) .
    \end{flalign*}
    See \cref{sec:sparse_domination} for the sparse domination of $B_j^{\mathcal{D}}(g,f)$.

    \subsubsection{Far \texorpdfstring{: $P \not\subset 3R$}{}} %
    In this case $\operatorname{d}(P,R) > \ell R > \ell P$, so the hypotheses of \cref{lemma:adapted_TH} are satisfied.
    After Cauchy--Schwarz, rearrange the sum using the common ancestor $K$, then
    let $\ell P = 2^{-m} \ell R = 2^{-m-i}\ell K$ and
    estimate the decay factor as in \eqref{eq:bound_the_coefficients}:    
    \begin{align*}
      \sum_{R \in \mathcal{D}} \int_R \lvert g\rvert \sum_{\substack{P : \ell P < \ell R \\ \operatorname{d}(P,R) > \ell R}} \langle f , h_P \rangle^2 \frac{(\sqrt{\ell P \ell R})^\alpha}{D(P,R)^{\alpha + d}}
      \le \sum_{i,m} \sum_{K \in \mathcal{D}} \sum_{\substack{ R \subset K \\ \ell R = 2^{-i}\ell K}} \int_R \lvert g\rvert \sum_{\substack{ P \subset K \\ \ell P = 2^{-m-i}\ell K}} \langle f,h_P\rangle^2 \frac{(\sqrt{\ell P \ell R})^\alpha}{\operatorname{d}(P,R)^{\alpha + d}} \\
      \lesssim_r \sum_{i\in\mathbb{N}} 2^{-i\alpha/2} \sum_{m\in\mathbb{N}} \sum_{K \in \mathcal{D}} \int_K \lvert g\rvert \sum_{\substack{ P \subset K \\ \ell P = 2^{-m-i}\ell K}} \langle f,h_P\rangle^2 \frac{2^{-(m+i)\alpha/4} 2^{-i\alpha/2}}{\lvert K \rvert} \\
      \le \sum_{i\in\mathbb{N}} 2^{-i\alpha/2} \sum_{j \in \mathbb{N}} 2^{-j\alpha/4} \sum_{K \in \mathcal{D}} \fint_K \lvert g\rvert \sum_{\substack{ P \subset K \\ \ell P = 2^{-j}\ell K}} \langle f,h_P\rangle^2 
    \end{align*}
    where $j \coloneqq m +i$ and we bounded by the sum over all $j\ge 0$, since all terms are non-negative.
    After summing in $i$, what is left is bounded by $B_j^{\mathcal{D}}(g,f)$.
    This concludes this case and the reduction of \eqref{term1} to a dyadic form. \qed

    \section{Reduction of \texorpdfstring{\eqref{term2}}{(II)} to a dyadic form}\label{sec:term2}
    In this section we prove the following bound
    \begin{equation}\label{eq:mind_the_gap_PR}
      \sum_{R \in \mathcal{D}} \iint_{W_R} \lvert g \rvert \Big\lvert \sum_{P\,:\,P \supset R^{(r)}} \theta_t( \Delta_P f \1_{P \setminus P_R} ) \Big\rvert^2 \frac{\D{t}}{t} \D{x} \lesssim B_0^{\mathcal{D}}(g,f).
    \end{equation}
    The dyadic form $B_0^{\mathcal{D}}(g,f)$ defined in \eqref{eq:def_B_0}
    is controlled by a sparse form in \cref{sec:sparse_domination}.
    \begin{remark}
      The goodness of $R$ gives the lower bound on the distance
      $\operatorname{d}(R,\partial P)> (\ell P)^{1-\gamma} (\ell R)^\gamma$.

      As will be clear from the proof,
      inequality \eqref{eq:mind_the_gap_PR} holds if one replaces
      the indicator $\1_{P \setminus P_R}$ with $\1_{K\setminus P_R}$ where $K$ is $\mathbb{R}^d$ or any other larger cube containing $P$.
    \end{remark}
    To prove \eqref{eq:mind_the_gap_PR}, we use
    a classical estimate for the Poisson kernel.
    \begin{lemma}[Poisson off-diagonal estimates]\label{lemma:off-diagonal-Poisson}
      Let $\beta\in(0,1],r\in\mathbb{N}$ and $\gamma$ as in the introduction
      and let $Q,P \in \mathcal{D}$ such that $Q^{(r)} \subset P$ and $Q$ is $r$-good.
      Then 
      \begin{equation*}%
        \int_{\mathbb{R}^d \setminus P} \frac{(\ell Q)^\beta}{\operatorname{d}(y,Q)^{\beta + d}} \D{y} \lesssim \left(\frac{\ell Q}{\ell P}\right)^\eta 
      \end{equation*}
      where $\eta = \beta - \gamma(\beta + d)$.
    \end{lemma}    
    \begin{proof}%
      Decompose $\mathbb{R}^d \setminus P$ in annuli $ A_k = 3^{k+1}P \setminus 3^k P $ for $k\in\mathbb{N}$.      
      Then on each annulus $\operatorname{d}(y,Q) > \operatorname{d}(\partial(3^k P), Q)$.
      Since $\ell P > 2^r \ell Q$, use the goodness of $Q$ to obtain the bound.
    \end{proof}
    \begin{proof}[Proof of \eqref{eq:mind_the_gap_PR}]
      When $(x,t) \in W_R$ the size condition \eqref{eq:size_condition} and \cref{lemma:off-diagonal-Poisson} give
      \begin{equation*}
        \theta_t( \Delta_P f \1_{P\setminus P_R})(x) \lesssim \lVert \Delta_P f \rVert_{L^\infty} \int_{P\setminus P_R} \frac{(\ell R)^\alpha}{(\ell R + \operatorname{d}(y,R) )^{\alpha + d}} \D{y}
        \lesssim  \frac{\lvert\langle f,h_P \rangle\rvert}{ \lvert P \rvert^{1/2}} \left(\frac{\ell R}{\ell P_R}\right)^\eta 
      \end{equation*}
      where $\eta = \alpha - \gamma(\alpha + d)>0$.
      The sum $\sum_{P \supset R^{(r)}} (\ell R/\ell P_R)^\eta$ is a geometric series. An application of Cauchy--Schwarz gives
      \begin{flalign*}
        \sum_{R \in \mathcal{D}} \iint_{W_R} \lvert g \rvert \left\lvert \sum_{P \supset R^{(r)}} \frac{\lvert\langle f,h_P \rangle\rvert}{ \lvert P \rvert^{1/2}} \left(\frac{\ell R}{\ell P_R}\right)^\eta \right\rvert^2 \frac{\D{t}}{t} \D{x} 
        & \le  \sum_{R \in \mathcal{D}} \sum_{P \supset R^{(r)}} \frac{\langle f, h_P \rangle^2}{\lvert P \rvert} \left(\frac{\ell R}{\ell P_R}\right)^\eta \int_R \lvert g(x)\rvert \D{x} \\
        & \lesssim \sum_{i \ge r+1} 2^{-i\eta} \sum_{P \in \mathcal{D}} \frac{\langle f, h_P \rangle^2}{\lvert P \rvert} \sum_{\substack{R \subset P \\ \ell R = 2^{-i}\ell P}} \int_R \lvert g \rvert \\
        & = \sum_{i \ge r+1} 2^{-i\eta} \sum_{P \in \mathcal{D}} \frac{\langle f, h_P \rangle^2}{\lvert P \rvert} \int_P \lvert g \rvert .
      \end{flalign*}
      We sum in $i$ and then we bound by the dyadic form $B_0^{\mathcal{D}}(g,f)$.
    \end{proof}
    
    \section{Reduction of \texorpdfstring{\eqref{term3a}}{(IIIa)} to a sparse form}\label{sec:term3a}
    In this section we prove that there exists $c>0$ and a sparse family $\mathscr{S} \subseteq \mathcal{D}$ such that
    \begin{equation}\label{eq:reduction_term3a}
      \eqref{term3a} \lesssim  \sum_{R\in\mathcal{D}} \langle\lvert g\rvert\rangle_R \iint_{W_R} \big\lvert\sum_{P\in\mathcal{P}_{R}} \langle \Delta_P f\rangle_{P_R}  \theta_t\1_{P_R} \big\rvert^2 \frac{\D{t}}t \D{x}
      \lesssim \sum_{j\in\mathbb{N}} 2^{-c j} B^{\mathcal{D}}_j (g,f) + \Lambda_{\mathscr{S}}(g,f)
    \end{equation}
    where $\Lambda_{\mathscr{S}}(g,f)= \sum_{S\in\mathscr{S}} \langle\lvert g\rvert\rangle_S \langle\lvert f\rvert\rangle_S^2 \lvert S\rvert$.
    We remind the reader that $P_R$ is the dyadic child of $P$ which contains $R$, and
    $\mathcal{P}_{R}$ is the collection of $P$ containing $ R^{(r)}$.    
    \begin{remark}[Bound on $a$]
      Recall that $a$ is the good part of $g$ in the Calderón--Zygmund decomposition 
      of \cref{prop:C-Z_r} with $\lambda = A\langle\lvert g\rvert\rangle_R$. So $\lVert a \rVert_\infty \le 2^{d(r+1)} A\langle\lvert g\rvert\rangle_R$ and the first inequality in \eqref{eq:reduction_term3a} follows.
    \end{remark}

    \subsection{Stopping cubes}\label{subsec:stopping_family}
    Given two functions $f$ and $g$ and a cube $Q\subseteq \mathbb{R}^d$, consider the collections:
    \begin{align*}
      \mathcal{A}_f(Q) &= \{ S\in\mathcal{D}, S \subset Q \, : \, \langle\lvert f \rvert \rangle_S > A \langle\lvert f \rvert \rangle_{Q} \} , \\
      \mathcal{A}_g(Q) &= \{ S\in\mathcal{D}, S \subset Q \, : \, \langle\lvert g \rvert \rangle_S > A \langle\lvert g \rvert \rangle_{Q} \} .
    \end{align*}
    Let $\mathcal{A}^\star(Q)$ be the maximal dyadic components of the set $\mathcal{A}(Q) =\mathcal{A}_f(Q)\cup\mathcal{A}_g(Q)$.

    The weak $(1,1)$ bound for the dyadic maximal function
    ensures that there exists a constant $A > 1$ such that $\lvert\mathcal{A}(Q)\rvert \le \frac12 \lvert Q \rvert$ and so
    \begin{equation*}
      \Big\lvert \bigcup_{S\in \mathcal{A}^\star(Q)} S \Big\rvert = \sum_{S\in\mathcal{A}^\star(Q)} \lvert S\rvert \le \frac12 \lvert Q\rvert .
    \end{equation*}

    Fix  $Q_0$ in $\mathcal{D}$ containing the support of $f$ and $g$.
    The stopping family $\mathscr{S}$ is defined iteratively:
    \begin{equation*}
      \mathscr{S}_0 \coloneqq Q_0, \qquad \mathscr{S}_{n+1} \coloneqq \bigcup_{Q \in \mathscr{S}_n} \mathcal{A}^\star(Q), \qquad \mathscr{S} \coloneqq \bigcup_{n \in \mathbb{N}} \mathscr{S}_n .
    \end{equation*}
    \begin{remark}
      The family $\mathscr{S}$ is $\frac12$-sparse, since for any $S\in\mathscr{S}$ the set $E_S \coloneqq S\setminus \bigcup_{S'\in \mathcal{A}^\star(S)} S'$
      has measure $\lvert E_S \rvert > \frac12\lvert S\rvert$ and $\{E_S\}_{S\in\mathscr{S}}$ are disjoint.
    \end{remark}
    In the same way, taking $\mathcal{A}^\star(Q) $ to be the maximal dyadic components of $\mathcal{A}_g(Q)$
    produces a sparse family that we denote with $\mathscr{S}_g$.
    It will be used later when only the stopping cubes related to $g$ are needed.
    
    For a given $Q\in\mathcal{D}$, denote by $\widehat{Q}$ the minimal stopping cube $S \in \mathscr{S}$ such that $S \supseteq Q$.

    For $S\in\mathscr{S}$ let $\Stop{S}$ be the family of dyadic cubes contained in $S$, but not in any $S'\in\mathcal{A}^\star(S)$
    \begin{equation*}%
      \Stop{S} \coloneqq \{ R\in\mathcal{D} \,:\, \widehat{R} = S \} .
    \end{equation*}
    Also, we define $\rStop{S} \coloneqq \{ R\in\mathcal{D} \,:\, \widehat{R^{(r)}} = S \}$.
    Note that the maximal cubes in $\rStop{S}$ are the $r$-grandchildren of $S$.
    See \cref{fig:StoppingTree} in the appendix.

    \subsection{Reduction to a telescoping sum}\label{subsec:telescoping}
    We follow the decomposition in \cite{LaceyMartikainen,Squarefungeneralmeasures} %
    where the sum $\sum_{P\in\mathcal{P}_{R}} \langle \Delta_P f\rangle_{P_R} \1_{P_R}$ is decomposed in a telescopic sum plus off-diagonal terms.
    The off-diagonal terms are then bounded by a sum of the dyadic forms $B^{\mathcal{D}}_j(g,f)$
    or directly by a sparse form.
    \newline    
    Given $S \in \mathscr{S}$ such that $S\supset P_R$, 
    the indicator function $\1_{P_R}$ can be written as $\1_S - \1_{S\setminus P_R}$.
    Recall that $\widehat{P_R}$ is the minimal stopping cube containing $P_R$. Then
    \begin{numcases}{\langle \Delta_Pf \rangle_{P_R} \1_{P_R} =}
      \langle \Delta_Pf \rangle_{P_R}\1_{\widehat{P_R}} - \langle \Delta_Pf \rangle_{P_R} \1_{\widehat{P_R} \setminus P_R}  & if $P_R \not\in \mathscr{S}$ \label{eq:decompose_P} \\
      \langle \Delta_{P}f \rangle_{P_R} \1_{\widehat{P_R}} = \1_{\widehat{P_R}} \langle f \rangle_{P_R} - \1_{\widehat{P_R}}\langle f \rangle_{P} \nonumber \\
      \phantom{\langle \Delta_{P}f \rangle_{P_R} \1_{\widehat{P_R}} } = (\1_{\widehat{P_R}} \langle f \rangle_{P_R} - \1_{\widehat{P}} \langle f \rangle_{P}) + \1_{\widehat{P} \setminus \widehat{P_R}}\langle f \rangle_{P}  & if $P_R \in \mathscr{S}$. \label{eq:decompose_PR_further}
    \end{numcases}

    The term $\langle \Delta_Pf \rangle_{P_R} \1_{\widehat{P_R} \setminus P_R}$ is supported away from $R$,
    so one can use off-diagonal estimates as in \eqref{eq:mind_the_gap_PR}.
    Also notice that in the bound \eqref{eq:mind_the_gap_PR} and in its proof
    one can replace $\lvert g\rvert$ by $\langle \lvert g\rvert \rangle_R$.
    In the same way, off-diagonal estimates are used for $\1_{\widehat{P} \setminus \widehat{P_R}}\langle f \rangle_{P}$
    as shown in \cref{lemma:mind_the_gap_S_setminus_PR} below.
    
    The terms $\langle \Delta_Pf \rangle_{P_R}\1_{\widehat{P_R}}$ and $\1_{\widehat{P_R}} \langle f \rangle_{P_R} - \1_{\widehat{P}} \langle f \rangle_{P}$
    left from \eqref{eq:decompose_P} and \eqref{eq:decompose_PR_further} 
    are rearranged to obtain a telescopic series. We have
    \begin{align*}
      \1_{\widehat{P_R}} \langle \Delta_P f\rangle_{P_R} =  \1_{\widehat{P_R}} \langle f \rangle_{P_R} & - \1_{\widehat{P_R}} \langle f \rangle_{P} 
      \qquad \text{ when }  P_R \not\in\mathscr{S} \\
      \text{ and } \quad  \1_{\widehat{P_R}} \langle f \rangle_{P_R} & - \1_{\widehat{P}} \langle f \rangle_{P}  \qquad \text{ when } P_R \in \mathscr{S} .
    \end{align*}
    If $P_R\not\in\mathscr{S}$ then
    $P$ and $P_R$ are  contained in the same minimal stopping cube $\widehat{P}$.
    So  $\widehat{P_R} = \widehat{P}$
    and the two cases add up to $ 2( \1_{\widehat{P_R}} \langle f \rangle_{P_R} - \1_{\widehat{P}} \langle f \rangle_{P} ) $   
    which leads to the telescopic sum
    \begin{equation}\label{eq:telescoping_sum}
      \sum_{\substack{P \in \mathcal{D} \\ R^{(r)} \subset P \subseteq Q_0}} \1_{\widehat{P_R}} \langle f \rangle_{P_R} - \1_{\widehat{P}} \langle f \rangle_{P} = \1_{\widehat{R^{(r)}}} \langle f \rangle_{R^{(r)}} - \1_{\widehat{Q_0}} \langle f\rangle_{Q_0}.
    \end{equation}
    
    Since $f$ is supported on a fixed $Q_0$, the average on larger cubes $Q^{(n)}_0$ containing $Q_0$ decreases:
    \begin{equation*}
      \langle f\rangle_{Q^{(n)}_0} = \frac{1}{\lvert Q^{(n)}_0 \rvert} \int_{Q_0} f \le \frac{1}{\lvert Q^{(n)}_0 \rvert}\lVert f \rVert_{L^1} \to 0 \quad \text{ as } n \to \infty . 
    \end{equation*}
    Thus when the sum in \eqref{eq:telescoping_sum} extends to all $P \supset R^{(r)}$,
    the term $\1_{\widehat{R^{(r)}}} \langle f \rangle_{R^{(r)}} $ is the only one remaining.

    We have then identified three terms
    \begin{equation*}
      \sum_{\substack{P\in\mathcal{D} \\ P \supset R^{(r)}}} \langle \Delta_P f\rangle_{P_R} \1_{P_R} =  \sum_{\text{telescopic}} - \sum_{\text{far}} + \sum_{\text{sparse}} 
    \end{equation*}
    where
    \begin{gather*}
      \sum_{\text{far}} \coloneqq \sum_{P\,:\,P \supset R^{(r)}} \langle \Delta_Pf \rangle_{P_R} \1_{\widehat{P_R} \setminus P_R} \,, \qquad 
      \sum_{\text{sparse}} \coloneqq \sum_{\substack{P\,:\,P\supset R^{(r)} \\ P_R \in \mathscr{S}}} \1_{\widehat{P} \setminus P_R}\langle f \rangle_{P} \\
      \text{ and } \quad \sum_{\text{telescopic}} \coloneqq \sum_{P \supset R^{(r)}} 2( \1_{\widehat{P_R}} \langle f \rangle_{P_R} - \1_{\widehat{P}} \langle f \rangle_{P} ) = \1_{\widehat{R^{(r)}}} \langle f \rangle_{R^{(r)}} .
    \end{gather*}
    Since the case with $\sum_{\text{far}}$ is done in \eqref{eq:mind_the_gap_PR},
    we show how to deal with the remaining two cases.
    
    \subsection{Bound by a sparse form}
    We bound the operator applied to $\1_{\widehat{R^{(r)}}} \langle f \rangle_{R^{(r)}} $ and $\1_{\widehat{P} \setminus P_R}\langle f \rangle_{P}$.

    \begin{lemma} Let $\mathscr{S}$ be the sparse collection defined in \cref{subsec:stopping_family}, then
      \begin{equation*}%
        \sum_{R \in \mathcal{D}} \iint_{W_R} \langle \lvert g\rvert \rangle_R \lvert \theta_t \1_{\widehat{R^{(r)}}}(x) \rvert^2 \langle f \rangle_{R^{(r)}}^2 \frac{\D{t}}{t} \D{x}
        \lesssim \sum_{S \in \mathscr{S}} \langle \lvert g\rvert \rangle_{S} \langle \lvert f\rvert \rangle_{S}^2  \lvert S \rvert .
      \end{equation*}
    \end{lemma}
    \begin{proof}%
      The set $\{\rStop{S} \colon S \in \mathscr{S}\}$ is a partition of $\mathcal{D}$, 
      so we write
      \begin{align*}
        \sum_{R \in \mathcal{D}} \iint_{W_R} \langle \lvert g \rvert \rangle_R \lvert \theta_t \1_{\widehat{R^{(r)}}}(x) \rvert^2 \langle f \rangle_{R^{(r)}}^2 \frac{\D{t}}{t}\D{x}
        & = \sum_{S \in \mathscr{S}} \sum_{R\,:\,\widehat{R^{(r)}} = S} 2^{rd} \langle \lvert g\rvert \rangle_{R^{(r)}} \langle f \rangle_{R^{(r)}}^2 \iint_{W_R} \lvert \theta_t \1_S (x)\rvert^2 \frac{\D{t}}{t} \D{x} \\
        & \lesssim_{r,d} \sum_{S \in \mathscr{S}} \langle \lvert g\rvert \rangle_{S} \langle \lvert f\rvert \rangle_{S}^2 \sum_{R\,:\,R \subset S} \iint_{W_R} \lvert \theta_t \1_S (x)\rvert^2 \frac{\D{t}}{t} \D{x} \\
        & = \sum_{S \in \mathscr{S}} \langle \lvert g\rvert \rangle_{S} \langle \lvert f\rvert \rangle_{S}^2  \int_S \int_0^{\ell S}  \lvert \theta_t \1_S (x)\rvert^2 \frac{\D{t}}{t} \D{x} \\
        & \le \testingC \sum_{S \in \mathscr{S}} \langle \lvert g\rvert \rangle_{S} \langle \lvert f\rvert \rangle_{S}^2  \lvert S \rvert 
      \end{align*}
      where we used the stopping conditions for $f$ and $g$, and
      the testing condition \eqref{eq:testing_condition}.
    \end{proof}
    
    \begin{lemma}\label{lemma:mind_the_gap_S_setminus_PR}
      Let $\mathscr{S}$ be the sparse collection defined in \cref{subsec:stopping_family}, then      
      \begin{equation}\label{eq:mind_the_gap_S_setminus_PR}
        \sum_{R \in \mathcal{D}} \iint_{W_R} \Bigg\lvert \sum_{\substack{P:P \supset R^{(r)} \\P_R \in \mathscr{S}}} \theta_t( \1_{\widehat{P} \setminus P_R} ) \langle f \rangle_P \Bigg\rvert^2 \lvert g\rvert \frac{\D{t}}{t} \D{x}
        \lesssim \sum_{S \in \mathscr{S}'} \langle \lvert f \rvert \rangle_S^2 \langle \lvert g \rvert \rangle_S \lvert S \rvert 
      \end{equation}
      where $\mathscr{S}'$ is the sparse collection of dyadic parents of $\mathscr{S}$.
    \end{lemma}
    \begin{proof}%
      Since $P \supset R^{(r)}$, the dyadic child $P_R =R^{(k)}$ for some integer $k \ge r$.
      For $(x,t) \in W_R$, an application of
      Poisson off-diagonal estimates (\cref{lemma:off-diagonal-Poisson}) gives
      \begin{equation*}
        \theta_t(\1_{\widehat{P} \setminus P_R})(x) = \theta_t(\1_{\widehat{R^{(k+1)}} \setminus R^{(k)}})(x) \lesssim (\ell R / \ell R^{(k)})^{\eta} = 2^{-k\eta}.
      \end{equation*}     
      After applying Cauchy--Schwarz the sums are rearranged using $P$ as the common ancestor:
      \begin{align*}
        \sum_{R \in \mathcal{D}} \int_R \lvert g \rvert \sum_{\substack{P\,:\,P \supset R^{(r)} \\ \widehat{P_R} = P_R}} \langle f\rangle_P^2 \left(\frac{\ell R}{\ell P_R}\right)^\eta 
                                & =  \sum_{k \ge r} 2^{-k\eta} \sum_{\substack{P \in \mathcal{D}\\ \text{with } P_R\in\mathscr{S}}} \langle f \rangle_P^2 \sum_{\substack{R\,:\,R \subset P \\ \ell R = 2^{-k-1}\ell P}} \int_R \lvert g \rvert \\
                                & =  \sum_{k \ge r} 2^{-k\eta} \sum_{P: P_R\in\mathscr{S}} \langle f \rangle_P^2 \int_{P} \lvert g \rvert \\
                                & \le \sum_{P: P_R\in\mathscr{S}} \langle \lvert f \rvert \rangle_P^2 \int_P \lvert g \rvert .
      \end{align*}
      Let $\mathscr{S}'$ be the collection $ \{ P \in\mathcal{D}\,:\, P \supset S,\, \ell P = 2\ell S \text{ for some } S \in \mathscr{S}\} $.
      If $\mathscr{S}$ is $\tau$-sparse, then $\mathscr{S}'$ is $\tau 2^{-d}$-sparse.
      This establishes \eqref{eq:mind_the_gap_S_setminus_PR} and concludes the proof.
    \end{proof}
    The sparse collection in \eqref{eq:reduction_term3a} can be taken
    as the union of $\mathscr{S}'$ and the stopping family in \cref{subsec:stopping_family}.

    \section{Reduction of \texorpdfstring{\eqref{term3b}}{(IIIb)} to a sparse form}\label{sec:term3b}
    In this section we show that there exists $c>0$ and a sparse family $\widetilde{\mathscr{S}}$ such that
    \begin{flalign*}      
      \sum_{R\in\mathcal{D}} \iint_{W_R} \Big\lvert \sum_{P\in\mathcal{P}_{R}} \theta_t\langle \Delta_P f\rangle_{P_R} \1_{P_R} \Big\rvert^2 \sum_{\substack{Q\in\mathcal{D}\\ Q \text{ good}, Q\subset R}} \Delta_Q b(x) \frac{\D{t}}t \D{x}
      \lesssim \sum_{j\in\mathbb{N}} 2^{-cj} B^{\mathcal{D}}_j (g,f) + \Lambda_{\widetilde{\mathscr{S}}}(g,f) .
    \end{flalign*}    
    In order to exploit  the goodness of $Q$, for example via Poisson off-diagonal estimates as in \cref{lemma:off-diagonal-Poisson},
    we need a gap of at least $r$ generations between $Q$ and $P_R$.
    This motivates the Calderón--Zygmund decomposition in \cref{prop:C-Z_r}.
    In particular, since $b$ is the bad part of $g$ at height $\lambda= A\langle\lvert g\rvert\rangle_R$ given by \cref{prop:C-Z_r},
    we have that
    \begin{equation*}
      \sum_{\substack{Q\in\mathcal{D}\\ Q\subset R}} \Delta_Q b = \sum_{L\in\mathcal{L}} \sum_{\mU{L}\in\ch_r(L)} \sum_{\substack{Q\in\mathcal{D}\\Q\subseteq \mU{L}}} \Delta_Q b_{\mU{L}}.
    \end{equation*}      
    Since $A>1$, the cubes in $\mathcal{L}$ are strictly contained in $R$.
    If we choose the constant $A$ as in the construction of the stopping family in \cref{subsec:stopping_family},
    then the cubes in $\mathcal{L}$ are also stopping cubes in $\mathscr{S}_g$.
    We can regroup the dyadic cubes $Q \subseteq \mU{L}$ in the stopping trees $\rStop{S}$ for all $S\in\mathscr{S}_g$ inside $R$.
    \begin{equation*}
      \sum_{L\in\mathcal{L}} \sum_{\mU{L}\in\ch_r(L)} \sum_{\substack{Q\in\mathcal{D}\\Q\subseteq \mU{L}}} \Delta_Q b_{\mU{L}} = \sum_{\substack{S\in\mathscr{S}_g\\ S\subset R}} \sum_{\mU{S}\in\maxStop{S}} \, \sum_{\substack{Q\in\rStop{S}\\Q\subseteq \mU{S}}} \Delta_Q b_{\mU{S}} .
    \end{equation*}
    The last sum is the Haar projection of $b$ on $\mathrm{Span}\{h_Q \,:\, Q\in\rStop{S}, Q\subseteq \mU{S}\}$.
    We denote this quantity by
    \begin{equation*}
      \EuScript{P}_{\mU{S}}(b) \coloneqq \sum_{\substack{Q\in\rStop{S}\\Q\subseteq \mU{S}}} \Delta_Q b_{\mU{S}} .
    \end{equation*}
    \begin{remark}  
      The Haar projection $\EuScript{P}_{\mU{S}}b$ is supported on $\mU{S}$ and equals $\EuScript{P}_{\mU{S}}(\lvert g\rvert)$.
      Indeed $b_{\mU{S}} = \1_{\mU{S}} (\lvert g\rvert - \langle\lvert g\rvert\rangle_{\mU{S}})$
      and for $Q \subseteq \mU{S}$ the Haar coefficient $\langle b_{\mU{S}},h_Q\rangle= \langle \lvert g\rvert, h_Q\rangle$.
    \end{remark}    

    We have then proved the following identity
    \begin{equation*}
      \eqref{term3b} =  \sum_{S\in\mathscr{S}_g} \sum_{\substack{\mU{S}\in\maxStop{S}\\ \mU{S} \text{ good}}}
      \sum_{R\,: R\supset S} \iint_{W_R} \big\lvert\sum_{P\in\mathcal{P}_{R}} \theta_t\langle \Delta_P f\rangle_{P_R} \1_{P_R} \big\rvert^2 \EuScript{P}_{\mU{S}}(\lvert g\rvert) \frac{\D{t}}t \D{x} .
    \end{equation*}
    With a slight abuse of notation, we omit the subscript in the stopping family $\mathscr{S}_g$ in the following.

    \begin{remark}[Estimates for the Haar projection]
      The Haar projection $\EuScript{P}_{\mU{S}}(\lvert g\rvert)$ has zero average and
      \begin{equation}\label{eq:bound_Haar-projector}
        \lVert \EuScript{P}_{\mU{S}} g \rVert_{L^1} \lesssim \lvert \mU{S}\rvert \langle\lvert g\rvert\rangle_S.      
      \end{equation}
      A proof of \eqref{eq:bound_Haar-projector} is in \cref{subsec:Haar-proj_on_maximal_cubes}.
      In particular, summing over all $\mU{S} \in \maxStop{S}$ gives
      \begin{equation}\label{eq:est_max_Haar}
        \sum_{\mU{S}\in\maxStop{S}} \lVert\EuScript{P}_{\mU{S}}g\rVert_{L^1} \lesssim  \sum_{\mU{S}\in\maxStop{S}} \lvert \mU{S}\rvert \langle \lvert g\rvert\rangle_S \le \int_S\lvert g\rvert .
      \end{equation}
    \end{remark}

    \subsection{Recover decay and telescopic sum}\label{subsec:recover_decay}

    Let $P_S$ be the dyadic child of $P$ containing $S$. Then
    \begin{equation*}%
      \sum_{P:P \supset R^{(r)}} \langle \Delta_P f \rangle_{P_R} \1_{P_R} = \sum_{P:P \supset S} \langle \Delta_P f \rangle_{P_S} \1_{P_S} - \sum_{P: S \subset P \subseteq R^{(r)}} \langle \Delta_P f \rangle_{P_S} \1_{P_S} .
    \end{equation*}
    The second term can be handled as in the subscale case (\cref{sec:P_small}),
    while the first can be reduced to a telescopic sum which equals $\langle f\rangle_S \1_S$.

    If one tries to reduce $\langle \Delta_P f \rangle_{P_S} \1_{P_S}$ to a telescopic term plus off-diagonal terms as in \cref{subsec:telescoping},
    the off-diagonal factor which should provide decay is the quantity
    \begin{equation*}
      \int_{\mathbb{R}^d\setminus P_S} \frac{(\ell R)^\alpha}{\operatorname{d}(y,\mU{S})^{\alpha+d}} \D{y}.
    \end{equation*}
    Here  the scale (numerator) and the distance (denominator) 
    don't match and \cref{lemma:off-diagonal-Poisson}
    seems unable to provide enough decay
    in order to handle the integral \emph{and} the sum over $R$.
    But the zero average property of $\EuScript{P}_{\mU{S}}(g)$ comes to the rescue
    bringing a factor $(\ell \mU{S})^{\alpha/2}$ at the numerator
    by exploiting the smoothness condition
    of the kernel. We will explain how.

    Let $x_{\mU{S}}$ be the centre of the $\mU{S}$ and consider the sublinear operator
    \begin{equation*} %
      K_t^{\mU{S}}f(x) \coloneqq \int_{\mathbb{R}^d} \frac{(t\lvert x-x_{\mU{S}}\rvert)^{\alpha/2}}{(t + \lvert x-y\rvert)^{\alpha+d}} \lvert f(y)\rvert \D{y} 
    \end{equation*}
    Since the Haar projection $\EuScript{P}_{\mU{S}}(g)$ is supported on $\mU{S}$, we have the following bound.
    \begin{lemma}\label{lemma:factorisation} Let $\mU{S}$ and $R$ be dyadic cubes with $\mU{S} \subset R$, then
      \begin{equation}\label{eq:factorisation}
        \iint_{W_R} \lvert \theta_t  f(x) \rvert^2 \frac{\D{t}}{t} \EuScript{P}_{\mU{S}}(g)(x) \D{x}
        \lesssim \iint_{W_R} \left( K_t^{\mU{S}} f(x) \right)^2 \frac{\D{t}}{t} \lvert \EuScript{P}_{\mU{S}}(g)(x)\rvert  \D{x}.
      \end{equation}
    \end{lemma}
    \begin{proof}%
      The idea is to use
      the zero average of $\EuScript{P}_{\mU{S}}(g)$ to exploit the smoothness condition \eqref{eq:smooth_condition}.      
      We recall that $\EuScript{P}_{\mU{S}}(g)$ is supported on $\mU{S} \subset R$.
      Consider the operator
      \begin{equation*}
        \mathcal{K}f(x) \coloneqq \int_{\ell R/2}^{\ell R} \left\lvert \int k_t(x,y) f(y)\D{y} \right\rvert^2 \frac{\D{t}}{t}
      \end{equation*}
      so that the left hand side of \eqref{eq:factorisation} equals $\int \mathcal{K}f(x) \EuScript{P}_{\mU{S}} g(x) \D{x}$.
      Let $x_{\mU{S}}$ be the centre of $\mU{S}$. Then
      \begin{flalign*}
        \int \mathcal{K}f(x) \EuScript{P}_{\mU{S}} g(x) \D{x} = \int \big(\mathcal{K}f(x) - \mathcal{K}f(x_{\mU{S}}) \big) \EuScript{P}_{\mU{S}} g(x) \D{x} 
      \end{flalign*}
      and the difference $\mathcal{K}f(x) - \mathcal{K}f(x_{\mU{S}})$ can be factorised as
      \begin{align*}
                          & \int_{\ell R/2}^{\ell R} \left\lvert \int k_t(x,y)f(y) \D{y} \right\rvert^2 - \left\lvert \int k_t(x_{\mU{S}},y)f(y) \D{y} \right\rvert^2 \frac{\D{t}}{t} \\
                                                    & = \int_{\ell R/2}^{\ell R} \left( \int [k_t(x,y) - k_t(x_{\mU{S}},y)]f(y) \D{y} \right) \left( \int [k_t(x,y) + k_t(x_{\mU{S}},y)]f(y) \D{y} \right) \frac{\D{t}}{t}\\
                                                    & \eqqcolon  \int_{\ell R/2}^{\ell R} \mathcal{K}_{\mU{S}}^-f(x) \cdot \mathcal{K}_{\mU{S}}^+f(x) \frac{\D{t}}{t} .
      \end{align*}
      For $x\in \mU{S}$, since $\mU{S}\subset R$ and $t\in(\ell R/2,\ell R)$, the distance $\lvert x - x_{\mU{S}} \rvert \le \ell \mU{S}/2 < \ell R/2 < t$, so
      by conditions \eqref{eq:smooth_condition}  and \eqref{eq:size_condition} we have
      \begin{align*}
        \mathcal{K}_{\mU{S}}^-f(x) \cdot \mathcal{K}_{\mU{S}}^+f(x) & \lesssim \int \frac{\lvert x-x_{\mU{S}} \rvert^\alpha}{(t + \lvert x-y\rvert)^{\alpha+d}} \lvert f(y)\rvert \D{y} \cdot \int \frac{t^\alpha}{(t + \lvert x-y\rvert)^{\alpha+d}} \lvert f(y)\rvert \D{y} \\
                                                                    & = \left( \int  \frac{(t\lvert x-x_{\mU{S}}\rvert)^{\alpha/2}}{(t + \lvert x-y\rvert)^{\alpha+d}} \lvert f(y)\rvert \D{y} \right)^2 \eqqcolon \big(K_t^{\mU{S}}f(x)\big)^2.
      \end{align*}       
    \end{proof}

    The operator $K_t^{\mU{S}}$ satisfies Poisson-like off-diagonal estimates.
    \begin{lemma}[Estimates for $K_t^{\mU{S}}$]\label{lemma:estimates_for_K}
      Let $x\in\mU{S}\subset R$ and $ t\in(\ell R/2,\ell R)$. Let $Q\in \mathcal{D}$ such that $Q\supset \mU{S}$.
      Then there exists $\eta>0$ such that the following estimates hold:
      \begin{align*}
        K_t^{\mU{S}} \1_{\mathbb{R}^d\setminus Q}(x) \lesssim \left(\frac{\ell \mU{S}}{\max(\ell R^{(r)}, \ell Q)}\right)^\eta , &&
        K_t^{\mU{S}} \1_{Q}(x) \lesssim \frac{\lvert Q\rvert}{\lvert R\rvert} \left(\frac{\ell \mU{S}}{\ell R}\right)^{\alpha/2} .
      \end{align*}
    \end{lemma}
    \begin{remark}
      Notice that the first estimate is better than the one in \cref{lemma:off-diagonal-Poisson}
      on smaller scale (when $\ell Q < \ell R^{(r)}$).
      For the second one,
      since $\ell \mU{S} < \ell R$,  we can also estimate
      \begin{equation*}
        K_t^{\mU{S}} \1_{Q}(x) \lesssim \frac{\lvert Q\rvert}{\lvert R\rvert} 
      \end{equation*}
      provided that $x\in\mU{S}$ and $ t\in(\ell R/2,\ell R)$.
    \end{remark}

    \begin{proof}[Proof of \cref{lemma:estimates_for_K}]
      For the second estimate, by forgetting the distance in the denominator, we simply have
      \begin{align*}
        K_t^{\mU{S}}(\1_{Q})(x) \lesssim \int_{Q} \frac{(\ell \mU{S} \ell R)^{\alpha/2}}{(\ell R+\operatorname{d}(y,\mU{S}))^{\alpha+d}} \D{y}
        & \le \frac{\lvert Q\rvert}{\lvert R\rvert} \left(\frac{\ell \mU{S}}{\ell R}\right)^{\alpha/2}.
      \end{align*}

      For the first estimate, when $Q \supset R^{(r)}$
      use $(a+b)^\alpha = (a+b)^{2\alpha/2} \ge (2 ab)^{\alpha/2}$ in order to apply off-diagonal estimates.
      For $x\in\mU{S}$ and $ t\in(\ell R/2,\ell R)$ we bound
      \begin{flalign}
        K_t^{\mU{S}}\1_{\mathbb{R}^d\setminus Q}(x) & \lesssim \int_{\mathbb{R}^d\setminus Q} \frac{(\ell \mU{S} \ell R)^{\alpha/2}}{(\ell R+\operatorname{d}(y,\mU{S}))^{\alpha+d}} \D{y} \nonumber \\
        & \lesssim  \int_{\mathbb{R}^d \setminus Q} \frac{(\ell \mU{S} \ell R)^{\alpha/2}}{(\ell R \cdot\operatorname{d}(y,\mU{S}))^{\alpha/2}} \frac{ \D{y} }{\operatorname{d}(y,\mU{S})^{d}}
        = \int_{\mathbb{R}^d \setminus Q} \frac{(\ell \mU{S})^{\alpha/2}}{\operatorname{d}(y,\mU{S})^{\alpha/2+d}} \D{y}. \label{eq:reduction_alpha/2}
      \end{flalign}
      Then apply \cref{lemma:off-diagonal-Poisson} with $\beta = \alpha/2$
      \begin{flalign*}
        \int_{\mathbb{R}^d \setminus Q} \frac{(\ell \mU{S})^{\alpha/2}}{\operatorname{d}(y,\mU{S})^{\alpha/2+d}} \D{y} \lesssim \left(\frac{\ell \mU{S}}{\ell Q}\right)^\eta .
      \end{flalign*}
      
      When $\mU{S} \subset Q\subset R^{(r)}$, split $\1_{\mathbb{R}^d\setminus Q}$ as $ \1_{\mathbb{R}^d\setminus R^{(r)}} + \1_{R^{(r)}\setminus Q}$.
      Estimate $ K_t^{\mU{S}}(\1_{\mathbb{R}^d\setminus R^{(r)}})$ as in \eqref{eq:reduction_alpha/2}. Then applying \cref{lemma:off-diagonal-Poisson} with $\beta=\alpha/2$ gives
      \begin{equation*}
        K_t^{\mU{S}}(\1_{\mathbb{R}^d\setminus R^{(r)}})(x) \lesssim \left(\frac{\ell \mU{S}}{\ell R^{(r)}}\right)^{\eta}
      \end{equation*}
      where $\eta$ is positive and equals $\frac{\alpha}2 - \gamma(\frac{\alpha}2 +d) < \frac{\alpha}2$.
      For $K_t^{\mU{S}}(\1_{R^{(r)}\setminus Q})$ we bound
      \begin{flalign*}
        K_t^{\mU{S}}(\1_{R^{(r)}\setminus Q})(x) & \lesssim \int_{R^{(r)}} \frac{(\ell \mU{S} \ell R)^{\alpha/2}}{(\ell R)^{\alpha/2} (\ell R)^{\alpha/2+d}} \D{y} \\
        & \le \frac{\lvert R^{(r)}\rvert}{\lvert R\rvert} \left(\frac{\ell \mU{S}}{\ell R}\right)^{\alpha/2} \lesssim_{r,d} \left(\frac{\ell \mU{S}}{\ell R}\right)^{\alpha/2} = 2^{r\alpha/2} \left(\frac{\ell \mU{S}}{\ell R^{(r)}}\right)^{\alpha/2}.
      \end{flalign*}
      Adding the two bounds gives
      \begin{align*}
        K_t^{\mU{S}}\1_{\mathbb{R}^d\setminus Q}(x)
        & \lesssim  \left(\frac{\ell \mU{S}}{\ell R^{(r)}}\right)^{\eta} + \left(\frac{\ell \mU{S}}{\ell R^{(r)}}\right)^{\alpha/2} \le 2 \left(\frac{\ell \mU{S}}{\ell R^{(r)}}\right)^{\eta}
      \end{align*}
      since $\ell \mU{S} < \ell R^{(r)}$ and $\min(\eta,\alpha/2) = \eta$.
    \end{proof}

    \subsection{Reduction to telescopic: different terms}
    Apply \cref{lemma:factorisation} with $\sum \langle \Delta_P f \rangle_{P_R} \1_{P_R}$ in place of $f$
    to obtain
    \begin{equation*}
      \eqref{term3b} \lesssim \sum_{S\in\mathscr{S}} \sum_{\substack{\mU{S}\in\maxStop{S}\\ \mU{S} \text{ good}}} \sum_{R\,:R\supset S} \iint_{W_R} \Bigg( K_t^{\mU{S}} \sum_{P\,:P \supset R^{(r)}}\langle \Delta_P f \rangle_{P_R} \1_{P_R}\Bigg)^2 \frac{\D{t}}{t} \lvert \EuScript{P}_{\mU{S}}g \rvert \D{x}.
    \end{equation*}    
   We split the sum in $P$ to obtain a telescopic sum  as in \cref{subsec:telescoping}, 
   with an extra subscale term:
   \begin{equation*}
     \sum_{P \supset R^{(r)}} \langle \Delta_P f \rangle_{P_R} \1_{P_R} = \sum_{\text{telescopic}} - \sum_{\text{far}} + \sum_{\text{sparse}} - \sum_{\text{subscale}} 
   \end{equation*}
   where
    \begin{align*}  
      \sum_{\text{telescopic}} & \coloneqq 
      2 \sum_{P\,:\,P \supset S}(\langle f \rangle_{P_S} \1_{\widehat{P_S}} - \langle f \rangle_{P} \1_{\widehat{P}}) = 2 \langle f\rangle_S \1_S  && \sum_{\text{subscale}} \coloneqq \sum_{\substack{P\,:\,P\subseteq R^{(r)}\\P \supset S}} \langle \Delta_P f \rangle_{P_S} \1_{P_S} \nonumber \\
      \sum_{\text{far}} & \coloneqq \sum_{P\,:\,P \supset S} \langle \Delta_P f \rangle_{P_S} \1_{\widehat{P_S} \setminus P_S} &&
      \sum_{\text{sparse}} \coloneqq  \sum_{\substack{P\,:\,P \supset S \\ P_S \in \mathscr{S}}} \langle f \rangle_P \1_{\widehat{P} \setminus P_S}.
    \end{align*}
    Then we bound
    \begin{equation*}
      \Big\lvert \sum_{\text{telescopic}} - \sum_{\text{far}} + \sum_{\text{sparse}} - \sum_{\text{subscale}} \Big\rvert
      \le \Big\lvert \sum_{\text{telescopic}}\Big\rvert + \Big\lvert\sum_{\text{far}}\Big\rvert + \Big\lvert\sum_{\text{sparse}}\Big\rvert + \Big\lvert\sum_{\text{subscale}} \Big\rvert .
    \end{equation*}
    We estimate $K_t^{\mU{S}}$ applied to each term by using sublinearity and \cref{lemma:estimates_for_K}.
    Then take the supremum in $t$ on the Whitney region $W_R$ to bound
    the remaining integral $\int_{\ell R/2}^{\ell R} \D{t}/t$ by $1$.

    We give the details in each case.
    
    \subsection{Telescopic term}
    This case is bounded by the sparse form $\Lambda_{\mathscr{S}}(f,g) = \sum_{S \in \mathscr{S}} \langle \lvert f\rvert \rangle_{S}^2 \int_S \lvert g\rvert$,
    where $\mathscr{S}$ is the stopping family of $g$.
    \begin{lemma} It holds that
      \begin{flalign*}
        \sum_{S\in\mathscr{S}} & \sum_{\mU{S}\in\maxStop{S}} \sum_{R\,:\,R\supset S} \iint_{W_R} \langle f\rangle_S^2 \left(K_t^{\mU{S}} \1_{S}\right)^2 \frac{\D{t}}{t} \lvert \EuScript{P}_{\mU{S}}g \rvert \D{x} \lesssim \Lambda_{\mathscr{S}}(f,g)
      \end{flalign*}
    \end{lemma}
    \begin{proof}
      For $x\in \mU{S}$ and $t\in(\ell R/2,\ell R)$ 
      we estimate $K_t^{\mU{S}}(\1_{S})(x) \lesssim \lvert S\rvert/\lvert R\rvert$ and $\int_{\ell R/2}^{\ell R} \D{t}/t \le 1$. Then
      by using \eqref{eq:est_max_Haar} for the Haar projection we have
      \begin{align*}
        \sum_{S \in \mathscr{S}} & \sum_{\mU{S} \in \maxStop{S}} \langle f \rangle_{S}^2 \sum_{R:R \supset S} \iint_{W_R} \left(K_t^{\mU{S}}\1_{S}(x)\right)^2 \frac{\D{t}}{t} \lvert \EuScript{P}_{\mU{S}}g(x)\rvert \D{x} \\
                                & \lesssim \sum_{S \in \mathscr{S}} \sum_{\mU{S} \in \maxStop{S}}\langle f \rangle_{S}^2 \sum_{R:R \supset S} \left(\frac{\lvert S\rvert}{\lvert R\rvert}\right)^2  \lVert\EuScript{P}_{\mU{S}}g\rVert_{L^1} \\
                                & \lesssim_{r,d}  \sum_{S \in \mathscr{S}} \langle f \rangle_{S}^2 \sum_{R:R \supset S} \left(\frac{\lvert S\rvert}{\lvert R\rvert}\right)^2  \int_S \lvert g\rvert \le \sum_{S \in \mathscr{S}} \langle \lvert f\rvert \rangle_{S}^2 \int_S \lvert g\rvert .
      \end{align*}
    \end{proof}
    
    \subsection{Subscale term}
    This term is bounded in a similar way as in the subscale case in \cref{subsec:subscale_inside}.     
    \begin{lemma} It holds that
      \begin{flalign*}
        \sum_{S\in\mathscr{S}} \sum_{\mU{S}\in\maxStop{S}} \sum_{R\,:\,R\supset S} \iint_{W_R} \Big(K_t^{\mU{S}} \sum_{\mathrm{subscale}}\Big)^2 \frac{\D{t}}{t} \lvert \EuScript{P}_{\mU{S}}g \rvert \D{x}
        \lesssim \sum_{j\in\mathbb{N}} 2^{-j\alpha/4} B_j^{\mathcal{D}}(g,f) .
    \end{flalign*}
  \end{lemma}
  \begin{proof}
    First, since $K_t^{\mU{S}}$ is sublinear, we bound
     \begin{flalign*}
        K_t^{\mU{S}} \Big( \sum_{\mathrm{subscale}}\Big) \le \sum_{\substack{P\,:\,P\subseteq R^{(r)}\\P \supset S}} \lvert\langle \Delta_P f \rangle_{P_S}\rvert K_t^{\mU{S}} ( \1_{P_S} ) .
    \end{flalign*}
    Then for $x\in \mU{S}$ and $t\in(\ell R/2,\ell R)$ 
    we estimate $K_t^{\mU{S}} \1_{P_S}$ using \cref{lemma:estimates_for_K}
    \begin{equation*}
      K_t^{\mU{S}} \1_{P_S}(x)  \lesssim \left(\frac{\ell P_S}{\ell R}\right)^d \left(\frac{\ell \mU{S}}{\ell R}\right)^{\alpha/2} .
    \end{equation*}    
    Bound $\ell P_S < \ell P$ and $\lvert \langle \Delta_P f \rangle_{P_S}\rvert \le \lvert \langle f , h_P \rangle\rvert \lvert P \rvert^{-1/2}$, then we apply Cauchy--Schwarz 
    \begin{align*} %
      & \sum_{S\in\mathscr{S}} \sum_{\mU{S} \in \maxStop{S}}  \sum_{R\,:\,R \supset S}
      \int_{R} \Bigg(\sum_{\substack{P\,:\,P\subseteq R^{(r)}\\P \supset S}} \lvert \langle f , h_P \rangle\rvert \frac{\lvert P \rvert^{1/2}}{\lvert R\rvert} \Bigg)^2 \left(\frac{\ell \mU{S}}{\ell R}\right)^{\alpha} \lvert\EuScript{P}_{\mU{S}}g\rvert \D{x} \\
      & \le \sum_{S\in\mathscr{S}} \sum_{\mU{S} \in \maxStop{S}} \sum_{R \supset S}
                               \Bigg(\sum_{\substack{P\,:\,P\subseteq R^{(r)}\\P \supset S}} \frac{\langle f , h_P \rangle^2 }{\lvert R\rvert} \Bigg) \Bigg(\sum_{P \subseteq R^{(r)}} \frac{\lvert P \rvert}{\lvert R\rvert} \Bigg) \left(\frac{\ell \mU{S}}{\ell R}\right)^{\alpha} \lVert\EuScript{P}_{\mU{S}}g\rVert_{L^1} \nonumber \\
                                                       & \le \sum_{S\in\mathscr{S}} \sum_{\mU{S} \in \maxStop{S}} \sum_{R \supset S} \left(\frac{\ell S}{\ell R}\right)^{\alpha/2}
      \Bigg(\sum_{\substack{P\,:P \supset S\\P\subseteq R^{(r)}}} \frac{\langle f , h_P \rangle^2 }{\lvert R\rvert} \left(\frac{\ell P}{\ell R}\right)^{\alpha/4} \Bigg) \left(\sum_{P \subseteq R^{(r)}} \frac{\lvert P \rvert}{\lvert R\rvert} \left(\frac{\ell P}{\ell R}\right)^{\alpha/4}\right) \lVert\EuScript{P}_{\mU{S}}g\rVert_{L^1} .\nonumber
    \end{align*}
    The second factor after Cauchy--Schwarz is controlled
    as in subscale case  in \cref{lemma:2nd_factor_finite} where $P\subset 3R$.
    Then bound $\lVert\EuScript{P}_{\mU{S}}g\rVert_{L^1}$ as in \eqref{eq:est_max_Haar} to obtain
    \begin{gather*}
      \sum_{S\in\mathscr{S}}  \int_S\lvert g\rvert \sum_{R\,:\,R \supset S} \left(\frac{\ell S}{\ell R}\right)^{\alpha/2}
       \sum_{P\,:\,S \subset P \subseteq R^{(r)}} \frac{\langle f , h_P \rangle^2 }{\lvert R\rvert} \left(\frac{\ell P}{\ell R}\right)^{\alpha/4} \\
      = \sum_{R\in\mathcal{D}} \frac{1}{\lvert R\rvert} \sum_{\substack{S\in\mathscr{S}\\ S\subset R}} \int_S\lvert g\rvert \left(\frac{\ell S}{\ell R}\right)^{\alpha/2}
      \sum_{P\,:\,S \subset P \subseteq R^{(r)}} \langle f , h_P \rangle^2  \left(\frac{\ell P}{\ell R}\right)^{\alpha/4} .      
    \end{gather*}
    For $i,j \in \mathbb{N}$, let $\ell P = 2^{-j}\ell R^{(r)}$ and $\ell S = 2^{-i}\ell R$. Extend the sum over all $P$ such that $P\subseteq R^{(r)}$ and rearrange 
    \begin{align*}
      \sum_{R\in\mathcal{D}} & \frac{1}{\lvert R\rvert} \sum_{\substack{S\in\mathscr{S}\\ S\subset R}} \int_S\lvert g\rvert \left(\frac{\ell S}{\ell R}\right)^{\alpha/2}
      \sum_{P\,:\,S \subset P \subseteq R^{(r)}} \langle f , h_P \rangle^2  \left(\frac{\ell P}{\ell R}\right)^{\alpha/4} \\
      & = \sum_{i,j} \sum_{R \in \mathcal{D}} 2^{-i\alpha/2} \frac{1}{\lvert R\rvert} \sum_{\substack{S \subset R \\ \ell S = 2^{-i}\ell R}} \int_S\lvert g\rvert  
      \sum_{\substack{P \subseteq R^{(r)} \\ \ell P = 2^{-j}\ell R^{(r)}}} \langle f , h_P \rangle^2 \left(\frac{\ell P}{ \ell R^{(r)}}\right)^{\alpha/4} 2^{r\alpha/4} \\
                & \lesssim_{r,d} \sum_{i,j} 2^{-i\alpha/2} 2^{-j\alpha/4} \sum_{R \in \mathcal{D}} \fint_{R} \lvert g \rvert  \sum_{\substack{P \subseteq R^{(r)} \\ \ell P = 2^{-j}\ell R^{(r)}}} \langle f , h_P \rangle^2  \\
                & \lesssim_{r,d} \sum_{j\in\mathbb{N}} 2^{-j\alpha/4}  \sum_{R^{(r)} \in \mathcal{D}} \fint_{R^{(r)}} \lvert g \rvert  \sum_{\substack{P \subseteq R^{(r)} \\ \ell P = 2^{-j}\ell R^{(r)}}} \langle f , h_P \rangle^2
      \le 3^d \sum_{j\in\mathbb{N}} 2^{-j\alpha/4} B_j^{\mathcal{D}}(g,f) . 
    \end{align*}
  \end{proof}
    
    \subsection{Far and Sparse terms}
    In this subsection we show that 
    \begin{flalign*}
       \sum_{S \in \mathscr{S}} \sum_{\mU{S} \in \maxStop{S}} \sum_{R\,:\,R \supset S}
      \iint_{W_R} \Big(K_t^{\mU{S}} \Big(\sum_{\text{far}} + \sum_{\text{sparse}}\Big) \Big)^2 \frac{\D{t}}{t} \lvert\EuScript{P}_{\mU{S}}g\rvert \D{x} 
      & \lesssim B_0^{\mathcal{D}}(g,f) + \Lambda_{\mathscr{S}}(f,g) .
    \end{flalign*}

    Since $K_t^{\mU{S}}$ is sublinear and positive, we bound
    \begin{flalign*}
      K_t^{\mU{S}} \Big(\sum_{\text{far}} + \sum_{\text{sparse}}\Big) & \le \sum_{P\,:\,P \supset S} \lvert\langle \Delta_P f \rangle_{P_S}\rvert  K_t^{\mU{S}}(\1_{\widehat{P_S}\setminus P_S})
      + \sum_{\substack{P:P \supset S\\P_S \in \mathscr{S}}} \lvert\langle f \rangle_{P}\rvert K_t^{\mU{S}}(\1_{\widehat{P}\setminus P_S} ) \\
      & \le  \sum_{P\,:\,P \supset S} \Big( \lvert \langle \Delta_P f \rangle_{P_S}\rvert + \lvert \langle f \rangle_P \rvert \1_{\{P_S\in\mathscr{S}\}} \Big) K_t^{\mU{S}}(\1_{\mathbb{R}^d\setminus P_S}) .
    \end{flalign*}
    Then split the sum over $P$ and consider the two cases:
    \begin{equation*}
      \sum_{P\,:\,P \supset S}  = \sum_{P:P \supset R^{(r)}} + \sum_{\substack{P\,:\,P \subseteq  R^{(r)}\\ P\supset S}} \eqqcolon (i) + (ii) .
    \end{equation*}    
    
    \begin{lemma}[Bound for $(i)$]\label{lemma:bound_far}
      Let $F_P$ be either $\langle \Delta_P f \rangle_{P_R}$ or $\langle f \rangle_P \1_{\{P_S\in\mathscr{S}\}}$. Then 
      \begin{equation*}
        \sum_{S \in \mathscr{S}} \sum_{\mU{S} \in \maxStop{S}} \sum_{R:R \supset S}
        \iint_{W_R} \Bigg(\sum_{P\,:\,P \supset R^{(r)}} \lvert F_P\rvert \cdot K_t^{\mU{S}}\1_{\mathbb{R}^d\setminus P_S}(x) \Bigg)^2 \frac{\D{t}}{t} \lvert\EuScript{P}_{\mU{S}}g\rvert \D{x}
        \lesssim \Lambda_{\mathscr{S}}(f,g) .
      \end{equation*}
    \end{lemma}

    \begin{lemma}[Bound for $(ii)$]\label{lemma:bound_ii+iii}
      Let $F_P$ be either $\langle \Delta_P f \rangle_{P_S}$ or $\langle f \rangle_P \1_{\{P_S\in\mathscr{S}\}}$. Then 
      \begin{equation*}
        \sum_{S \in \mathscr{S}} \sum_{\mU{S} \in \maxStop{S}} \sum_{R\,:\,R \supset S}
        \iint_{W_R} \Big(\sum_{\substack{P\,:\,P\supset S\\P \subseteq  R^{(r)}}} \lvert F_P\rvert \cdot K_t^{\mU{S}}\1_{\mathbb{R}^d\setminus P_S}(x) \Big)^2 \frac{\D{t}}{t} \lvert\EuScript{P}_{\mU{S}}g\rvert \D{x}
        \lesssim \Lambda_{\mathscr{S}}(f,g) .
      \end{equation*}
    \end{lemma}

        \begin{proof}[Proof of \cref{lemma:bound_far}]
          In this case $P \supset R \supset S$, so the dyadic child $P_S$ equals $P_R$.
          Using \cref{lemma:estimates_for_K}, since $\ell \mU{S} <\ell S$, we have
           \begin{equation*}
            K_t^{\mU{S}}\1_{\mathbb{R}^d\setminus P_R}(x) \lesssim \left(\frac{\ell \mU{S}}{\ell P_R}\right)^\eta
            \le \left(\frac{\ell S}{\ell R}\right)^{\eta} \left(\frac{\ell R}{\ell P_R}\right)^{\eta} .
          \end{equation*}
          We bound $\int_{\ell R/2}^{\ell R} \D{t}/t \le 1$ and then apply Cauchy--Schwarz         
         \begin{flalign*}
           \sum_{S \in \mathscr{S}} \sum_{\mU{S} \in \maxStop{S}} \sum_{R:R \supset S}
          \Bigg(\sum_{P\,:\,P \supset R^{(r)}} \lvert F_P\rvert \left(\frac{\ell S}{\ell R}\right)^{\eta} \left(\frac{\ell R}{\ell P_R}\right)^{\eta} \Bigg)^2 \lVert \EuScript{P}_{\mU{S}}g\rVert_{L^1} \\
          \le       \sum_{S \in \mathscr{S}} \sum_{\mU{S} \in \maxStop{S}} \sum_{R:R \supset S}                     
          \sum_{P\,:\,P \supset R^{(r)}} F_P^2 \left(\frac{\ell S}{\ell R}\right)^{2\eta} \left(\frac{\ell R}{\ell P_R}\right)^{\eta} \lVert\EuScript{P}_{\mU{S}}g\rVert_{L^1}
        \end{flalign*}
        since $\sum_{P\,:\,P \supset R^{(r)}} (\ell R/\ell P_R)^{\eta} \le 1$.
        Bound the sum of Haar projections as in \eqref{eq:est_max_Haar}
        \begin{gather*}
          \sum_{S \in \mathscr{S}} \sum_{\mU{S} \in \maxStop{S}} \sum_{R:R \supset S}                     
          \sum_{P:P \supset R^{(r)}} F_P^2 \left(\frac{\ell S}{\ell R}\right)^{2\eta} \left(\frac{\ell R}{\ell P_R}\right)^{\eta} \lVert\EuScript{P}_{\mU{S}}g\rVert_{L^1} \\
          \lesssim \sum_{S \in \mathscr{S}} \sum_{R:R \supset S}
          \sum_{P:P \supset R^{(r)}} F_P^2 \left(\frac{\ell S}{\ell R}\right)^{2\eta} \left(\frac{\ell R}{\ell P_R}\right)^{\eta} \int_S \lvert g\rvert .
        \end{gather*}
        Rearrange the sums
        \begin{equation*}
          \sum_{S \in \mathscr{S}} \sum_{\substack{R\in\mathcal{D}\\R \supset S}} \sum_{\substack{P\in\mathcal{D}\\P \supset R^{(r)}}}
          = \sum_{R\in\mathcal{D}} \sum_{P:P \supset R^{(r)}} \sum_{i\in\mathbb{N}} \sum_{\substack{S\in\mathscr{S}\\ S\subset R \\ \ell S=2^{-i}\ell R}}
        \end{equation*}
        then we continue as in the proof of \eqref{eq:mind_the_gap_PR}.
        \begin{flalign*}
          \sum_{R\in\mathcal{D}} & \sum_{P:P \supset R^{(r)}} F_P^2 \left(\frac{\ell R}{\ell P_R}\right)^{\eta} \sum_{i\in\mathbb{N}} \sum_{\substack{S\in\mathscr{S}\\ S\subset R \\ \ell S=2^{-i}\ell R}} \left(\frac{\ell S}{\ell R}\right)^{2\eta} \int_S \lvert g\rvert \\
          & \le \sum_{R\in\mathcal{D}} \sum_{P:P \supset R^{(r)}} F_P^2 \left(\frac{\ell R}{\ell P_R}\right)^{\eta} \sum_{i\in\mathbb{N}}2^{-i\eta} \int_R \lvert g\rvert \\
          & \le \sum_{P \in\mathcal{D}} F_P^2 \sum_{k\ge r} \sum_{\substack{R\,:\,R\subset P\\\ell R=2^{-k-1}\ell P}} \left(\frac{\ell R}{\ell P_R}\right)^{\eta} \int_R \lvert g\rvert \\
          & \le \sum_{P \in\mathcal{D}} F_P^2 \sum_{k\ge r} 2^{-k\eta} \int_P \lvert g\rvert \le \sum_{P \in\mathcal{D}} F_P^2 \int_P \lvert g\rvert .
        \end{flalign*}
        
        Now we distinguish the two cases for $F_P$. \vspace{5pt}
        
        \begin{minipage}[c]{0.4\textwidth}
          \[ \text{If } F_P = \langle \Delta_P f \rangle_{P_R} \]
          
          \begin{flalign*}
            \sum_{P \in\mathcal{D}} F_P^2 \int_P \lvert g\rvert
            & \le \sum_{P \in\mathcal{D}} \frac{\langle f,h_P \rangle^2}{\lvert P\rvert} \int_P \lvert g\rvert \\
            & \le 3^d  B_0^{\mathcal{D}}(g,f) .
          \end{flalign*}
          Then $ B_0^{\mathcal{D}}(g,f)$ is bounded by a sparse form in \cref{lemma:sparse_domination}.
        \end{minipage} %
        \quad \vline \quad
        \begin{minipage}[c]{0.4\textwidth}
          \[ \text{If } F_P = \langle f \rangle_P \1_{\{P_R\in\mathscr{S}\}} \]

          \begin{flalign*}
            \sum_{P \in\mathcal{D}} F_P^2 \int_P \lvert g\rvert
            & = \sum_{P\,:\,P_R\in\mathscr{S}} \langle f\rangle_P^2 \int_P \lvert g\rvert \\
            & = \Lambda_{\mathscr{S}'}(f,g)
          \end{flalign*}
          where $\mathscr{S}'$ is the sparse collection of dyadic parents of $\mathscr{S}$.
        \end{minipage}
        
      \end{proof}

      \begin{proof}[Proof of \cref{lemma:bound_ii+iii}]
        For $x\in \mU{S}$ and $t \in(\ell R/2, \ell R)$, since $\mU{S} \subset S \subset P \subseteq R^{(r)}$, by \cref{lemma:estimates_for_K} 
        \begin{equation*}
          K_t^{\mU{S}}(\1_{\mathbb{R}^d\setminus P_S})(x) \lesssim \left(\frac{\ell \mU{S}}{\ell R^{(r)}}\right)^{\eta} .
        \end{equation*}                
        Then  we distribute the decay factor which is bounded as following
        \begin{equation*}
          \left(\frac{\ell \mU{S}}{\ell R^{(r)}}\right)^\eta \le \left(\frac{\ell S}{\ell R^{(r)}}\right)^{\eta/2} \left(\frac{\ell S}{\ell P}\right)^{\eta/2}\left(\frac{\ell P}{\ell R^{(r)}}\right)^{\eta/2} .
        \end{equation*}
        Estimate the integral $\int_{\ell R/2}^{\ell R} \D{t}/t \le 1$ and the sum of Haar projections as in \eqref{eq:est_max_Haar}.       
        \begin{gather*}
          \sum_{S \in \mathscr{S}} \sum_{\mU{S} \in \maxStop{S}} \sum_{R \supset S}
          \iint_{W_R} \Bigg(\sum_{\substack{P:P\supset S\\P\subseteq R^{(r)}}} F_P \left(\frac{\ell \mU{S}}{\ell R^{(r)}}\right)^{\eta} \Bigg)^2 \frac{\D{t}}{t} \lvert\EuScript{P}_{\mU{S}}g\rvert \D{x} \\        
          \le \sum_{S\in\mathscr{S}} \sum_{\mU{S}\in\maxStop{S}}\lVert \EuScript{P}_{\mU{S}} g\rVert_{L^1} \sum_{R\,:\,R\supset S} \left(\frac{\ell S}{\ell R^{(r)}}\right)^{\eta} \Bigg(\sum_{\substack{P:P\supset S\\P\subseteq R^{(r)}}} F_P \left(\frac{\ell S}{\ell P}\right)^{\eta/2}\left(\frac{\ell P}{\ell R^{(r)}}\right)^{\eta/2} \Bigg)^2  \\
          \lesssim \sum_{S\in\mathscr{S}} \int_S\lvert g\rvert  \sum_{R\,:\,R\supset S} \left(\frac{\ell S}{\ell R^{(r)}}\right)^{\eta} \Bigg( \sum_{\substack{P:P\supset S\\P\subseteq R^{(r)}}}F_P \left(\frac{\ell S}{\ell P}\right)^{\eta/2}\left(\frac{\ell P}{\ell R^{(r)}}\right)^{\eta/2} \Bigg)^2 .
        \end{gather*}
        Apply Cauchy--Schwarz.
        \begin{flalign*}
          \sum_{S\in\mathscr{S}} & \int_S\lvert g\rvert  \sum_{R\,:\,R\supset S} \left(\frac{\ell S}{\ell R^{(r)}}\right)^{\eta} \Bigg(\sum_{\substack{P:P\supset S\\P\subseteq R^{(r)}}}F_P \left(\frac{\ell S}{\ell P}\right)^{\eta/2}\left(\frac{\ell P}{\ell R^{(r)}}\right)^{\eta/2} \Bigg)^2 \\
          & \le \sum_{S\in\mathscr{S}} \int_S\lvert g\rvert  \sum_{R\,:\,R\supset S} \left(\frac{\ell S}{\ell R^{(r)}}\right)^{\eta} \sum_{\substack{P:P\supset S\\P\subseteq R^{(r)}}} F_P^2\left(\frac{\ell S}{\ell P}\right)^{\eta} \cdot \sum_{\substack{P:P\supset S\\P\subseteq R^{(r)}}} \left(\frac{\ell P}{\ell R^{(r)}}\right)^{\eta}
        \end{flalign*}
        The last sum is finite:
          since $P \supset S$ there is only one ancestor for each generation.
          Since all terms are non-negative, we bound by removing the restriction $P\subset R^{(r)}$ in the sum in $P$.
          \begin{flalign*}
            \sum_{S\in\mathscr{S}} \int_S\lvert g\rvert  \sum_{R\,:\,R\supset S} \left(\frac{\ell S}{\ell R^{(r)}}\right)^{\eta} \sum_{P:P\supset S} F_P^2 \left(\frac{\ell S}{\ell P}\right)^{\eta}
            & \le \sum_{S\in\mathscr{S}} \int_S\lvert g\rvert  \sum_{P:P\supset S} F_P^2 \left(\frac{\ell S}{\ell P}\right)^{\eta} \\
            & = \sum_{P\in\mathcal{D}} F_P^2 \sum_{i\in\mathbb{N}} 2^{-i\eta} \sum_{\substack{S\in\mathscr{S}\\ S\subset P\\ \ell S=2^{-i}\ell P}} \int_S\lvert g\rvert  \\
            & \le \sum_{P\in\mathcal{D}} F_P^2 \int_P\lvert g\rvert \sum_{i\in\mathbb{N}} 2^{-i\eta} .
          \end{flalign*}
          The two cases for $F_P$ are as at the end of the proof of \cref{lemma:bound_far}.
        \end{proof}
            
    \section{Sparse domination of the dyadic form}\label{sec:sparse_domination}
    In this section we prove a sparse domination of the dyadic form $B_j^{\mathcal{D}}(g,f)$ defined in \eqref{eq:def_B_j}.
     
    Writing $1$ as $\langle \1_P \rangle_P$  we have
    \begin{equation*}
      B_j^{\mathcal{D}}(g,f) = \int_{\mathbb{R}^d} \sum_{K\in\mathcal{D}} \langle \lvert g\rvert \rangle_{3K}  \sum_{\substack{P\in\mathcal{D}\\P\subset 3K\\\ell P=2^{-j}\ell K}} \frac{\langle f, h_P \rangle^2}{\lvert P\rvert} \1_P(x) \D{x}. 
    \end{equation*}

    Let $Q_0$ be a dyadic cube containing the support of $f$ and $g$.
    On the complement of $Q_0$ the form is controlled.
    \begin{lemma}\label{lemma:control_away_from_Q_0}
      Let $B_j^{\mathcal{D}}\restriction_{Q_0^\complement}(g,f)$ be the restriction of $B_j^{\mathcal{D}}(g,f)$ to the complement $(Q_0)^{\complement}$, then
      \begin{equation*}
        B_j^{\mathcal{D}}\restriction_{Q_0^\complement}(g,f) \lesssim_d 2^{-j d} \langle \lvert g\rvert\rangle_{Q_0} \langle \lvert f\rvert \rangle_{Q_0}^2 \lvert Q_0 \rvert . 
      \end{equation*}
    \end{lemma}
    \begin{proof}
      Decompose $(Q_0)^{\complement}$ in the union of $Q_0^{(k+1)} \setminus Q_0^{(k)}$ for $k\in\mathbb{Z}_+$.
      The non-zero terms in $B_j^{\mathcal{D}}\restriction_{Q_0^\complement}(g,f)$ are the ones where $P$ intersects $Q_0$ and $(Q_0^{(k)})^\complement$.
      Then $P \supset Q_0^{(k)}$ and in particular $P = Q_0^{(m)}$ for $m > k$.
      There is only one ancestor for each $m$, so we have
      \begin{flalign*}
        B_j^{\mathcal{D}}\restriction_{Q_0^{(k+1)} \setminus Q_0^{(k)}}(g,f) & = \int_{Q_0^{(k+1)} \setminus Q_0^{(k)}} \sum_{K\in\mathcal{D}} \langle\lvert g\rvert\rangle_{3K} \sum_{\substack{P\subset 3K\\\ell P = 2^{-j}\ell K\\P\supset Q_0^{(k)}}} \Bigg(\frac{\langle f,h_P\rangle}{\lvert P\rvert^{1/2}}\Bigg)^2 \1_P(x) \D{x} \\
        & \lesssim \sum_{m=k+1}^\infty\langle\lvert g\rvert\rangle_{3Q_0^{(m+j)}} \langle\lvert f\rvert\rangle_{Q_0^{(m)}}^2 \lvert Q_0^{(m)}\rvert \\
        & = \sum_{m=k+1}^\infty 3^{-d} 2^{-(m+j)d} \langle\lvert g\rvert\rangle_{Q_0} 2^{-2md} \langle\lvert f\rvert\rangle_{Q_0}^2 2^{md} \lvert Q_0\rvert \\
        & \le 2^{-jd} \langle\lvert g\rvert\rangle_{Q_0}\langle\lvert f\rvert\rangle_{Q_0}^2 \lvert Q_0\rvert \sum_{m=k+1}^\infty 2^{-2md} .
      \end{flalign*}      
      The last sum is bounded by $2^{-k d}$ and summing over $k\in\mathbb{Z}_+$ concludes the proof.
    \end{proof}

    It's enough to construct a sparse family inside $Q_0$.
    Taking the supremum of $\langle\lvert g\rvert\rangle_{3K}$ over all $K\in\mathcal{D}$  we have
    \begin{equation*}
      B_j^{\mathcal{D}}(g,f) \le  \int M^{3\mathcal{D}}g(x) \cdot (S_j^{3\mathcal{D}} f(x))^2 \D{x}
    \end{equation*}
    where $M^{3\mathcal{D}}$ and $S_j^{3\mathcal{D}}$ denote the maximal function and the  square function given by
     \begin{equation}\label{eq:def_S3_and_M}
       M^{3\mathcal{D}}f \coloneqq \sup_{Q \in \mathcal{D}} \langle \lvert f \rvert \rangle_{3Q} \1_{3Q} , \qquad
       \big(S_j^{3\mathcal{D}} f(x)\big)^2 \coloneqq \sum_{R\in\mathcal{D}} \sum_{\substack{P\in\mathcal{D}\\P \subset 3R \\ \ell P = 2^{-j}\ell R}}\frac{\langle f,h_P\rangle^2 }{\lvert P \rvert} \1_P(x) .       
     \end{equation}
     As we see below, $S_j^{3\mathcal{D}}$ is pointwise
     controlled by the square function $S_j^{\mathcal{D}}f(x)$ given by
     \begin{equation*}%
       \big(S^{\mathcal{D}}_j f (x)\big)^2 \coloneqq \sum_{Q \in \mathcal{D}} \sum_{\substack{P\in\mathcal{D}\\P \subset Q \\ \ell P = 2^{-j} \ell Q}} \frac{\langle f , h_P \rangle^2 }{\lvert P \rvert} \1_P(x).
     \end{equation*}     

     \begin{prop}[Pointwise control]\label{prop:enlarged_control}
      Let $f \in L^2(\mathbb{R}^d)$ and $j \in \mathbb{N}_0$. For all $x \in \mathbb{R}^d$ it holds that
      \begin{equation*}
        S^{\mathcal{D}}_j f(x) \le S_j^{3\mathcal{D}}f(x) \le 3^{d/2} S^{\mathcal{D}}_j f(x)
      \end{equation*}
    \end{prop}
    \begin{proof}
      The enlarged cube $3R$ is the union of $3^d$ cubes $\{R_a\}_a$ in the same dyadic grid $\mathcal{D}$. So
      \begin{align*}
        \left(S_j^{3\mathcal{D}} f(x)\right)^2 &= \sum_{R\in\mathcal{D}} \sum_{\substack{P\,:\,P\subset 3R\\\ell P=2^{-j}\ell R}} \langle f,h_P\rangle^2 \frac{\1_P(x)}{\lvert P\rvert} \\
                   &= \sum_{R\in\mathcal{D}} \sum_{a=1}^{3^d} \sum_{\substack{P\,:\,P\subset R_a \\\ell P=2^{-j}\ell R_a}} \langle f,h_P\rangle^2 \frac{\1_P(x)}{\lvert P\rvert} \\
                   &=  \sum_{a=1}^{3^d} \left(S_j^{\mathcal{D}}f(x)\right)^2 \le 3^d \left(S_j^{\mathcal{D}}f(x)\right)^2.
      \end{align*}
    \end{proof}
    
    We show that the square function $S_j^{\mathcal{D}}$ satisfies a weak $(1,1)$ bound.
    The proof follows the one for dyadic shifts without separation of scales \cite[Theorem 5.2]{SharpDyadicShift}
    and \cite[Lemma 4.4]{LaceyMena}.
    \begin{prop}\label{prop:weak_bound_S}
      Let $j \in \mathbb{Z}_+$. There exists $C>0$ such that for any $f\in L^1(\mathbb{R}^d)$ it holds that
      \begin{equation*}
        \sup_{\lambda >0} \lambda \lvert \{ x\in\mathbb{R}^d\,:\, S_j^{\mathcal{D}} f(x) > \lambda \} \rvert \le C (1+j) \lVert f \rVert_{L^1} .        
      \end{equation*}
      In particular $ \lVert S_j^{\mathcal{D}} \rVert_{L^1 \to L^{1,\infty}}$ grows at most polynomially in $j$.
    \end{prop}
    
    \begin{proof}%
      First, $S_j^{\mathcal{D}}$ is bounded in $L^2$ with norm independent of $j$. %

      We want to show that for any $\lambda>0$ we have
      \begin{equation*}
        \lvert \{ x\in \mathbb{R}^d \,:\, S^{\mathcal{D}}_j f(x) > \lambda \} \rvert \le C \frac{\lVert f \rVert_1}{\lambda} . 
      \end{equation*}
      Let $f=g+b$ be the Calderón-Zygmund decomposition of $f$ at height $\lambda > 0$.
      Then $\lVert g \rVert_\infty \leq 2^d \lambda$ and in particular $\lVert g\rVert_2^2 \lesssim \lambda \lVert f\rVert_1 $, while
      $b = \sum_{Q \in \mathcal{L}} b_Q$, where $b_Q$ is supported on $Q$ and $\int b_Q =0$.
      The cubes $Q$ in $\mathcal{L}$ are maximal dyadic cubes 
      such that $\lambda < \langle\lvert f \rvert \rangle_Q \le 2^d \lambda$.
     
      Let $E$ be the union of the cubes in $\mathcal{L}$.
      Then $ \lvert E \rvert = \sum_{Q \in \mathcal{L}} \lvert Q \rvert \le \lambda^{-1} \lVert f \rVert_1 $ so
      it is enough to estimate the superlevel sets on the complement of $E$. Using the decomposition of $f$ we have
      \begin{align*}
        \lvert \{ x \in E^\complement \colon S_j^{\mathcal{D}}f(x) > \lambda \} \rvert &\le  \left\lvert \left\{ x \colon S_j^{\mathcal{D}} g(x) > \frac{\lambda}2 \right\} \right\rvert
        + \left\lvert \left\{ x \in E^\complement \colon S_j^{\mathcal{D}} b(x) > \frac{\lambda}2 \right\} \right\rvert \\
        & \lesssim_d \frac{\lVert f \rVert_1}{\lambda}
        +\frac{2}{\lambda} \lVert S_j^{\mathcal{D}} b \rVert_{L^1(E^{\complement})} .
      \end{align*}
      The last bound follows by using Chebyshev's inequality for the good part:
      \begin{equation*}
        \left\lvert \left\{ S_j^{\mathcal{D}} g > \frac{\lambda}2 \right\} \right\rvert \le \frac{4}{\lambda^2} \lVert S_j^{\mathcal{D}} g \rVert_2^2 \lesssim \frac{\lVert g \rVert_2^2}{\lambda^2} \lesssim \frac{\lVert f \rVert_1}{\lambda}
      \end{equation*}
      and Markov's inequality for the bad part. The sublinearity of $S_j^{\mathcal{D}}$ and the triangle inequality imply that
      \begin{equation*}
        \lVert S_j^{\mathcal{D}} b \rVert_{L^1(E^{\complement})} \le \sum_{Q\in\mathcal{L}} \lVert S_j^{\mathcal{D}} b_Q \rVert_{L^1(E^{\complement})} .
      \end{equation*}
      For each $Q\in\mathcal{L}$, only dyadic cubes $K\supset Q$ contribute to the norm $\lVert S_j^{\mathcal{D}} b_Q \rVert_{L^1(E^\complement)}$,
      since if $K\subseteq Q$, then $K$ would be inside $E$.
      Thus $K$ is an ancestor of $Q$, so $K=Q^{(k)}$ for some integer $k\ge 1$.
      For $k>j$ each $j$-child $P\subset K$ contains $Q$, and so $\langle b_Q ,h_P\rangle$ vanishes,
      by the zero average of $b_Q$. Thus we estimate
      \begin{align*}
        \lVert S_j^{\mathcal{D}} b_Q\rVert_{L^1(E^{\complement})} & \le \int_{E^{\complement}} \sum_{K\in\mathcal{D}} \sum_{\substack{P \subset K \\ \ell P=2^{-j}\ell K}} \lvert\langle b_Q,h_P \rangle\rvert \frac{\1_P(x)}{\lvert P \rvert^{1/2}} \D{x} \\
                                                            & \le \sum_{k=1}^j \sum_{\substack{K\supset Q \\ \ell K=\ell Q^{(k)}}} \sum_{\substack{P \subset K \\ \ell P =2^{-j} \ell K}} \lvert \langle b_Q,h_P \rangle\rvert  \lvert P \rvert^{1/2} \\
                                                            & \le \sum_{k=1}^j \sum_{\substack{K\supset Q \\ \ell K=\ell Q^{(k)}}} \sum_{\substack{P \subset K \\ \ell P =2^{-j} \ell K}} \lVert b_Q \rVert_{L^1(P)} \\
                                                            & \le \sum_{k=1}^j \sum_{\substack{K\supset Q \\ \ell K=\ell Q^{(k)}}} \lVert b_Q \rVert_{L^1(K)} = \sum_{k=1}^j \lVert b_Q \rVert_{L^1}.
      \end{align*}
      Since $\lVert b_Q \rVert_{L^1(K)} = \lVert b_Q \rVert_{L^1} \lesssim \lambda \lvert Q \rvert < \int_Q \lvert f\rvert$, and
      there is only one ancestor of $Q$ for each $k$, we have
      \begin{align*}
        \sum_{k=1}^j \lVert b_Q \rVert_{L^1} \lesssim \sum_{k=1}^j \int_Q\lvert f \rvert \le j \int_Q\lvert f\rvert .
      \end{align*}
      Summing over all $Q\in\mathcal{L}$ gives the bound
      \begin{equation*}
        \sum_{Q\in\mathcal{L}} \lVert S_j^{\mathcal{D}} b_Q \rVert_{L^1(E^{\complement})} \lesssim \sum_{Q\in\mathcal{L}} j \lVert f \rVert_{L^1(Q)} \le j \lVert f\rVert_{L^1(\mathbb{R}^d)} .
      \end{equation*}
    \end{proof}

    The operator $M^{3\mathcal{D}}$ defined in \eqref{eq:def_S3_and_M} is also  weak $(1,1)$ as it is
    bounded by the Hardy--Littlewood maximal function, which is weakly bounded.

    The following lemma exploits the weak boundedness of the operators $M^{3\mathcal{D}}$ and $S_j^{\mathcal{D}}$ to construct a sparse collection $\mathscr{S}$.
    The proof adapts the one in \cite[Lemma 4.5]{LaceyMena} to our square function. We include the details for the convenience of the reader.
    
    \begin{lemma}[Sparse domination of $B_j^{\mathcal{D}}$] \label{lemma:sparse_domination}
      Let $j\in \mathbb{Z}_+$. For any pair of compactly supported functions $f,g\in L^\infty(\mathbb{R}^d)$
      there exists a sparse collection $\mathscr{S}$ such that
      \begin{equation*}
        B_j^{\mathcal{D}}(g,f) \lesssim \int M^{3\mathcal{D}}g \cdot (S_j^{\mathcal{D}}f)^2 \lesssim (1+j)^2 \sum_{S \in \mathscr{S}} \langle \lvert f\rvert\rangle_S^2 \langle \lvert g\rvert\rangle_S \lvert S\rvert
      \end{equation*}
      where the implicit constant does not depend on $j$.
    \end{lemma}
    \begin{proof}[Proof of \cref{lemma:sparse_domination}]
      Fix a cube $Q_0 \in \mathcal{D}$ containing the union of the supports of $f$ and $g$.
      By \cref{lemma:control_away_from_Q_0} it is enough to construct a sparse family inside $Q_0$.
      Consider the set
      \begin{equation*}
        F(Q_0) \coloneqq  \{ x \in Q_0 \, : \,M^{3\mathcal{D}}g(x) > C \langle \lvert g\rvert \rangle_{Q_0} \} \cup \{ x \in Q_0 \, : \, S_j^{\mathcal{D}} f(x) > C(1+j)\langle \lvert f\rvert \rangle_{Q_0} \} .
      \end{equation*}
      By the weak boundedness of $M^{3\mathcal{D}}$ and $S_j^{\mathcal{D}}$, there exists $C>0$ such that $ \lvert F(Q_0) \rvert \le \frac12 \lvert Q_0 \rvert$.
      Then
      \begin{align*}
        \int_{Q_0} M^{3\mathcal{D}}g \cdot (S_j^{\mathcal{D}} f)^2 & \le \int_{Q_0\setminus F(Q_0)} M^{3\mathcal{D}}g \cdot (S_j^{\mathcal{D}} f)^2 + \int_{F(Q_0)} M^{3\mathcal{D}}g \cdot (S_j^{\mathcal{D}} f)^2 \\
        & \le C^3 (1+j)^2 \langle \lvert g\rvert \rangle_{Q_0}\langle \lvert f\rvert \rangle_{Q_0}^2 \lvert Q_0 \rvert + \sum_{Q\in \mathcal{F}} \int_{Q} M^{3\mathcal{D}}g \cdot (S_j^{\mathcal{D}}f)^2 
      \end{align*}
      where $\mathcal{F}$ is the collection of maximal dyadic cubes covering $F(Q_0)$.
      Iterating on each $Q\in\mathcal{F}$ produces a sparse family of cubes $\mathscr{S}$, since
       $\{E_Q \coloneqq Q \setminus F(Q)\}_{Q\in\mathscr{S}}$ are pairwise disjoint
       and $\lvert E_Q \rvert > \frac12 \lvert Q\rvert$ for each $Q$ in $\mathscr{S}$.
     \end{proof}   

    \section{Proofs for the reduction to a dyadic form}\label{sec:proofs_for_reduction}
    \begin{proof}[Proof of \cref{lemma:LM}]
      We distinguish three cases: $\ell P > 2^r \ell R$, where the goodness is used;
      $\ell P~\in~[\ell R, 2^r\ell R]$, where $\ell P $ and $\ell R$ are comparable;
      and $\ell P < \ell R$, where we use the zero-average of $\Delta_P f$
      and the regularity condition \eqref{eq:smooth_condition}.

      \begin{description}
      \item[($\ell P > 2^r\ell R$)]
        Using the size condition \eqref{eq:size_condition} and taking the supremum in $(x,t)\in W_R$
      \begin{align}
        \theta_t (\Delta_P f)(x) & \lesssim \int_P \frac{t^\alpha}{(t + \lvert x - y \rvert)^{\alpha+d}} \lvert \Delta_Pf(y)\rvert \D{y} \nonumber \\
                              & \le \lVert \Delta_P f \rVert_{L^1} \frac{(\ell R)^\alpha}{(\frac{\ell R}{2} + \operatorname{d}(R,P))^{\alpha+d}} \label{eq:use_size}.
      \end{align}
      If $\operatorname{d}(R,P) > \ell P$, since $\ell P > 2^r\ell R$  the conclusion follows.
      Otherwise, by the goodness of $R$, we have that $\ell P < \operatorname{d}(R,P) \left(\frac{\ell P}{\ell R}\right)^{\gamma}$.
      The same bound holds for $\ell R$, so      
      \begin{equation*}
        D(P,R)^{\alpha+d} < 3^{\alpha+d} \operatorname{d}(P,R)^{\alpha+d} \left( \frac{\ell P}{\ell R} \right)^{\gamma(\alpha+d)} 
      \end{equation*}
      which implies
      \begin{equation*}
        D(P,R)^{\alpha+d} \left( \frac{\ell R}{\ell P} \right)^{\alpha/2} \lesssim_{\alpha,d} \operatorname{d}(P,R)^{\alpha+d} \left( \frac{\ell R}{\ell P} \right)^{-\gamma(\alpha+d)+\alpha/2} \le \operatorname{d}(P,R)^{\alpha+d}
      \end{equation*}
      since $\ell R/\ell P < 1$ and  $\alpha/2 - \gamma(\alpha+d)$ is non-negative for $\gamma \le \frac{\alpha}{2(\alpha+d)}$.
      Then multiply and divide \eqref{eq:use_size} by $D(P,R)^{\alpha+d} (\ell P)^{-\alpha/2} (\ell R)^{\alpha/2}$ 
      to conclude.
      
      \item[($\ell R \le \ell P \le 2^r\ell R$)] The lengths of $P$ and $R$ are comparable
      and the conclusion follows. %

      \item[($\ell P <\ell R$)] Let $x_P$ be the centre of $P$. Then
      \begin{align*}
        \int k_t(x,y) \Delta_Pf(y) \D{y}
        = & \int ( k_t(x,y) - k_t(x,x_P) ) \Delta_Pf(y) \D{y} \\
        \lesssim & \int \frac{\lvert y - x_P\rvert^\alpha}{(t + \lvert x - y \rvert)^{\alpha+d}} \lvert\Delta_Pf(y)\rvert \D{y}
      \end{align*}
       by the smoothness condition \eqref{eq:smooth_condition},
      since $\lvert y - x_P \rvert \le \frac{\ell P}2 < \frac{\ell R}2 < t$.
      To conclude, note that
      \begin{equation*}
        \frac{(\ell P)^\alpha}{(\frac{\ell R}{2} + \operatorname{d}(R,P))^{\alpha+d}}
        < \frac{(\ell P)^\alpha}{(\frac{\ell R}{4} + \frac{\ell P}{4} + \operatorname{d}(R,P))^{\alpha+d}}
        \le 4^{\alpha+d} \frac{(\sqrt{\ell R \ell P})^\alpha}{D(R,P)^{\alpha+d}}.
      \end{equation*}
    \end{description}
  \end{proof}

  \subsection{Counting close cubes}
  In both cases ``\ref{near}'' and ``\ref{close}'',
  given a fixed $R$ we estimate the number of $P$ such that $3P \supset R$.  
  \begin{lemma}\label{lemma:counting3P}
    For $k\in\mathbb{N}$ let
      $\mathscr{P}_k(R) \coloneqq \{ P\,:\,3P\supset R,\ell P = 2^k \ell R\}$.
      Then $\rvert \mathscr{P}_k(R)\rvert = 3^d$.
  \end{lemma}
  \begin{proof}
    Let $R^{(k)}$ be the $k$-ancestor of $R$. Then $R^{(k)}$ belongs to $\mathscr{P}_k(R)$.
    There are $3^d-1$ cubes $P$ adjacent to $R^{(k)}$ with $\ell P = \ell R^{(k)}$.
    Each of them is such that $3P \supset R^{(k)}$, so in particular $3P \supset R$.

    On the other hand, if $P$ is not adjacent to $R^{(k)}$ and $\ell P = \ell R^{(k)}$
    then $\operatorname{d}(P,R^{(k)}) \ge \ell P$, so $3P$ does not contain $R^{(k)}$, nor $R$.

    This shows that the $P$ in $ \mathscr{P}_k(R)$ are exactly the cubes contained in $3R^{(k)}$ with $\ell P = \ell R^{(k)}$, and
    there are $3^d$ of such cubes.
  \end{proof}
  
    \begin{proof}[Proof of \cref{lemma:2nd_factor_finite}] We present each case separately.
      \begin{description}
      \item[\ref{far} ]
        $\ell P \ge 2^{r+1}\ell R$ and $\operatorname{d}(P,R) > \ell P$. The largest term in $D(P,R)$ is $\operatorname{d}(P,R)$.
        Fix $R$ and $k \in \mathbb{N}$.
        Given $m \in \mathbb{N}$ there are at most $2^{m d}$ cubes $P$ with length $2^k \ell R$
        such that $2^m \ell P < \operatorname{d}(P,R) \le 2^{m+1}\ell P$, so rearranging the sum
        \begin{align*}
          \sum_{\substack{P \,:\, \ell P \ge \ell R \\ \operatorname{d}(R,P) > \ell P}}
          \frac{(\sqrt{\ell R \ell P})^\alpha}{\operatorname{d}(R,P)^{\alpha+d}} \lvert P \rvert 
          & = \sum_{m = 1}^\infty \sum_{k=r}^\infty \sum_{\substack{P\,:\,\ell P = 2^{k+1} \ell R \\ 2^{m+1} \ge \operatorname{d}(P,R)/\ell P > 2^m }}
          \left(\frac{\sqrt{ \ell R \ell P}}{\operatorname{d}(R,P)}\right)^{\alpha} \left( \frac{\ell P }{\operatorname{d}(P,R)}\right)^d \\
          & \le \sum_{m = 1}^\infty \sum_{k=0}^\infty 2^{md} 2^{-\alpha(k/2 + m)} 2^{-m d}
            \le \sum_{k,m} 2^{-\alpha(k/2 + m)} .
        \end{align*}

      \item[\ref{near} ] For $P$ such that $3P \setminus P \supset R$ and $\ell P \ge 2^{r+1} \ell R$, the decay comes from $\operatorname{d}(P,R)$,
        which is bounded below by $\ell P (\ell R/\ell P)^\gamma$, and  $\gamma = \alpha/(4\alpha+4d)$.  Then
        \begin{align*}
          \sum_{\substack{P \,:\, 3P\setminus P \supset R \\ \ell P \ge 2^{r+1}\ell R}}
          \frac{(\sqrt{\ell R \ell P})^\alpha}{\operatorname{d}(R,P)^{\alpha+d}} \lvert P \rvert
          &= \sum_{k=r+1}^\infty \sum_{\substack{P \,:\, 3P\setminus P \supset R \\ \ell P = 2^k \ell R}} \frac{\lvert P \rvert}{\operatorname{d}(P,R)^d} \left(\frac{\sqrt{\ell P \ell R}}{\operatorname{d}(P,R)}\right)^\alpha 
          \lesssim_d \sum_{k=r+1}^\infty 2^{-k\alpha/4}
        \end{align*}
        where, by \cref{lemma:counting3P}, the $P$ in the sum are at most $3^d$ for each $k$.

      \item[\ref{close} ] For $\ell R\le \ell P \le \ell R^{(r)}$ and $3P \supset R$, the leading term in the long-distance is $\ell R$.        
        \begin{equation*}
          \frac{(\sqrt{\ell R \ell P})^\alpha}{D(R,P)^{\alpha+d}} \lvert P \rvert \le \frac{(2^{r/2} \ell R)^\alpha}{(\ell R)^{\alpha}} \frac{\lvert P\rvert}{\lvert R\rvert} \le 2^{\alpha r/2} \frac{\lvert P\rvert}{\lvert R\rvert}.
        \end{equation*}
        To estimate the term we will fix a scale $k$ for $P$, such that $0\le k \le r$, then
        \begin{align*}
          \sum_{\substack{P \,:\,3P\supset R\\\ell R\le \ell P \le \ell R^{(r)}}}
          \frac{(\sqrt{\ell R \ell P})^\alpha}{D(R,P)^{\alpha+d}} \lvert P \rvert
          & \lesssim \sum_{\substack{P\,:\,3P\supset R\\\ell R\le \ell P \le \ell R^{(r)}}} \frac{\lvert P \rvert}{\lvert R\rvert}
          = \sum_{k=0}^r \sum_{\substack{P\,:\,3P\supset R\\\ell P = 2^k \ell R}} 2^{kd} \\
          & \le 2^{rd} \sum_{k=0}^r \lvert \{ P\,:\,3P\supset R,\ell P = 2^k \ell R\}\rvert
          \le 2^{rd} 3^d (r+1) .
        \end{align*}
        Where to estimate the number of $P$ we used \cref{lemma:counting3P}.
        
      \item[\ref{subscale}, $P\subset 3R$ ]The leading term in the long-distance $D(R,P)$ is again $\ell R$.
        For any $k \in \mathbb{N}$, there are $3^d 2^{k d}$ cubes $P$ such that $P \subset 3R$ and $ 2^k \ell P = \ell R$, so 
        \begin{align*}
          \sum_{P \,:\, P \subset 3R}
          \frac{(\sqrt{\ell R \ell P})^\alpha}{D(R,P)^{\alpha+d}} \lvert P \rvert
          & \le \sum_{P \subset 3R} \left(\frac{\ell P}{\ell R}\right)^{\alpha/2 + d}
          = \sum_{k = 1}^\infty \sum_{\substack{P \subset 3R \\ \ell P = 2^{-k} \ell R}} 2^{-k\frac{\alpha}2} 2^{-k d} \lesssim_d \sum_{k = 1}^\infty 2^{-k\frac{\alpha}2} < \infty .
        \end{align*}
      \item[\ref{subscale}, $P \not\subset 3R$ ]  In this case $\operatorname{d}(P,R) > \ell R > \ell P$.
        Regroup the $P$ according to length and distance:
        \begin{align*}
          \sum_{\substack{P \,:\,P \not\subset 3R \\ \ell P < \ell R}}
          \frac{(\sqrt{\ell R \ell P})^\alpha}{D(R,P)^{\alpha+d}} \lvert P \rvert 
          & = \sum_{k \in \mathbb{N}} \sum_{\substack{P \,:\, 2^k \ell P = \ell R \\ \operatorname{d}(P,R) > \ell R}} 2^{-kd} \left(\frac{\ell R }{D(P,R)}\right)^d 2^{-k\alpha/2} \left(\frac{\ell R}{D(P,R)}\right)^\alpha \\
          & \le  \sum_{k,m} \sum_{\substack{P \,:\, 2^k \ell P = \ell R \\ 2^{m+1} \ge \operatorname{d}(P,R) / \ell R > 2^m}} 2^{-k(d+\alpha/2)} 2^{-m d} 2^{-m\alpha} \le  \sum_{k,m} 2^{-k\alpha/2} 2^{-m\alpha} .
        \end{align*}
        This because there are at most $2^{m d}$ cubes $R$ in the range given by the distance,
        which means at most $2^{m d} \cdot 2^{k d}$ cubes $P$ with $\ell P = 2^{-k}\ell R$.
      \end{description}
    \end{proof}

    \begin{proof}[Proof of \cref{lemma:adapted_TH} (for $\ell P < \ell R$)]
      Recall that $\gamma \in (0,\frac12)$.
      Let $K$ be the minimal cube $K \supset R$ such that $\ell K \ge 2^r \ell R$ and $\operatorname{d}(P,R) \le \ell K \left(\frac{\ell P}{\ell K}\right)^\gamma$.
      (The set of such cubes is not empty since $\ell K \left(\frac{\ell P}{\ell K}\right)^\gamma$ equals $\ell P \left(\frac{\ell K}{\ell P}\right)^{1-\gamma}$ which goes to infinity as $\ell K \to \infty$.)      
      First, observe that $P \subset K$. Suppose not, then 
      \begin{equation*}
        \ell K \left(\frac{\ell P}{\ell K}\right)^\gamma < \ell K \left(\frac{\ell R}{\ell K}\right)^\gamma < \operatorname{d}(R,\partial K) \overset{P \subset K^{\complement}}{\le} \operatorname{d}(R,P)
      \end{equation*}
      which is absurd because of the second condition on $K$.
      It remains to show the upper bound for $\ell K$. By minimality of $K$, one of the following conditions holds: either
      \begin{equation*}
        \frac{\ell K}2 < 2^r \ell R \quad \text{ or } \quad \frac{\ell K}2 \left(\frac{\ell P}{\frac12 \ell K}\right)^\gamma < \operatorname{d}(P,R).
      \end{equation*}
      Since by hypothesis $\operatorname{d}(P,R) > (\ell R)^{1-\gamma} (\ell P)^\gamma$,
      the first implies
      \begin{equation*}
        \ell K \left(\frac{\ell P}{\ell K}\right)^\gamma \le 2^r \ell R \left(\frac{\ell P}{\ell K}\right)^\gamma \le 2^r \ell R \left(\frac{\ell P}{\ell R}\right)^\gamma < 2^r \operatorname{d}(P,R) .
      \end{equation*}
      The latter gives: $\ell K (\ell P/ \ell K)^\gamma < 2 \operatorname{d}(P,R) \le 2^r \operatorname{d}(P,R)$.
    \end{proof}

    \appendix   
    
    \section{Conditional expectation and Haar projections}\label{apx:old_tricks}
    In this appendix we recall some known bounds for the Haar projection.
    These involve conditional expectation and martingales related to the Haar system, see also \cite[\S 6.4]{GrafakosClassical}.

    Let $\mathscr{S}$ be the stopping family   defined in \cref{subsec:stopping_family}.
    Given $S\in\mathscr{S}$, let $\mathcal{A}^\star(S)$ be the maximal stopping cubes inside $S$.
    Let $\mathscr{G}_S$ be the $\sigma$-algebra generated by $\mathcal{A}^\star(S)$.
    A function is measurable with respect the $\sigma$-algebra $\mathscr{G}_S$ if and only if   it is constant on any cube in $\mathcal{A}^\star(S)$.

    \subsection{Conditional expectation}
    Denote by $\mathbb{E}[\,\cdot\, | \mathscr{G}_S ]$ the projection on the space of measurable functions with respect to
    the $\sigma$-algebra $\mathscr{G}_S$.
    \begin{equation*}
      \mathbb{E}[ f | \mathscr{G}_S ](x) =
      \begin{dcases*}
        f(x) & if $x \in S \setminus \mathcal{A}(S)$ \\
        \langle f \rangle_{S'} & if $x \in S'$ for some $S' \in \mathcal{A}^\star(S)$.
      \end{dcases*}
    \end{equation*}
    For more details about this operator, we refer the reader to \cite[\S 2.6]{Martingales_LPtheory}.
    Let $\mathscr{S}$ be a stopping family for $f$.
    The supremum of $ \mathbb{E}[ f | \mathscr{G}_S ]$ in $S$ is either $f(x)$ (if $\mathcal{A}(S)$ is empty), or $\langle f \rangle_{S'}$ for some $S' \in \mathcal{A}^\star(S)$.
    In both cases   $ \lVert  \mathbb{E}[ f \1_S | \mathscr{G}_S ] \rVert_{L^\infty(S)} \lesssim_d \langle f \rangle_S $, since $\langle f \rangle_{S'} \le 2^d A \langle f \rangle_S$
    by the stopping conditions.

    \subsection{Haar projection}
    Given $S \in \mathscr{S}$, let  $\Stop{S} = \{ Q \in \mathcal{D} \, :\, \widehat{Q} = S\}$ be the collection of cubes $Q$
    such that $S$ is the minimal stopping cube containing $Q$.

    The Haar projection on $S$ is given by
    \begin{equation*}
      \EuScript{P}_S f \coloneqq \sum_{I \in \Stop{S}} \Delta_I f = \sum_{I\in\Stop{S}} \sum_{\epsilon \in \{0,1\}^d \setminus \{0\}^d} \langle f,h_I^\epsilon \rangle h_I^\epsilon
    \end{equation*}
    where $ \{h_I^\epsilon\}_{\epsilon}$ are the Haar functions on $I$.
    Being a sum of Haar functions on cubes in $\Stop{S}$, the Haar projection $\EuScript{P}_S f$
    is constant on any $S'\in\mathcal{A}^\star(S)$, so it's measurable on $\mathscr{G}_S$.
    It also holds that $ \EuScript{P}_S f = \EuScript{P}_S \mathbb{E}[ f \1_S | \mathscr{G}_S ] $.
    
    The Haar projection $\EuScript{P}_S f$ can be seen as
    a martingale transform, and so it satisfies the following
    \begin{lemma}[$L^p$ bound for martingale transform \cite{MR744226}]
      For $1 < p < \infty$ we have
      \begin{equation}\label{eq:Lp_bound_martingale}
        \lVert \EuScript{P}_S \mathbb{E}[f\1_S | \mathscr{G}_S] \rVert_p \le C_p \lVert \mathbb{E}[f\1_S | \mathscr{G}_S] \rVert_p .
      \end{equation}
    \end{lemma}  
    Combining \eqref{eq:Lp_bound_martingale} with the estimate for the supremum of $\mathbb{E}[f\1_S | \mathscr{G}_S]$ one obtains that
    \begin{equation*}
      \lVert \EuScript{P}_S f\rVert_p \lesssim_p \langle f\rangle_S .
    \end{equation*}

    \subsection{Richer \texorpdfstring{$\sigma$-}{sigma}-algebras and \texorpdfstring{$r$}{r}-Haar projections}
    The same idea works with slight modifications when $S$ is the minimal stopping cube containing the $r$-ancestor of $Q$.
    Let $\rStop{S}$ be the collection of cubes $Q$ such that $\widehat{Q^{(r)}} = S$.
    Define the $r$-Haar projection on $S$ as
    \begin{equation*}
      \EuScript{P}_S^r f = \sum_{Q \in \rStop{S}} \Delta_Q f.
    \end{equation*}
    \begin{remark}
      The projection $\EuScript{P}_{S}^{r} f$ is \emph{not} measurable on $\mathscr{G}_S$ in general,
      but it is measurable with respect to the richer $\sigma$-algebra generated by the $r$-grandchildren of $S'\in\mathcal{A}^\star(S)$, which is
      \begin{equation*}
        \mathscr{G}_{S}^r \coloneqq \sigma \Big( \{ (S')_r \in \ch_r(S'), S' \in \mathcal{A}^\star(S) \} \Big) .
      \end{equation*}  
    \end{remark}
    Then $\EuScript{P}_{S}^{r} f = \EuScript{P}_{S}^{r} \mathbb{E}[ f \1_S | \mathscr{G}_{S}^r ]$
    and we have the following
    \begin{lemma} \label{lemma:bound_sup_expectation}
      Given a function $f$, let $S$ be a stopping cube in $\mathscr{S}_f$ as defined in \cref{subsec:stopping_family}. Then 
      \begin{equation*}
        \lVert  \mathbb{E}[ f \1_S | \mathscr{G}_{S}^r ] \rVert_{L^\infty(S)} \lesssim_{d,r} \langle f \rangle_S .
      \end{equation*}
    \end{lemma}
    \begin{proof}
      Either $\lvert f(x) \rvert \le A \langle f \rangle_S$ for all $x \in S$, or there exists $S' \in \mathcal{A}^\star(S)$ with $x_0 \in (S')_r$
      and $\mathbb{E}[ f \1_S | \mathscr{G}_{S}^r ](x_0)  = \langle f \rangle_{(S')_r} $.
      Let $P$ be the dyadic parent of $(S')_r$. Then $P\in \rStop{S}$ and we have
      \begin{equation*}
        \langle f \rangle_{(S')_r} \le 2^d \langle f \rangle_{P} \le 2^d 2^{d r} \langle f \rangle_{P^r} \le 2^{d(r+1)} A\langle f \rangle_S 
      \end{equation*}
      where we used the stopping condition in the last inequality.
    \end{proof}

    \subsection{Haar projection on maximal cubes}\label{subsec:Haar-proj_on_maximal_cubes}
    For $S\in\mathscr{S}$, the $r$-grandchildren $\maxStop{S}$ are the maximal cubes in $\rStop{S}$.
    Then the restriction of Haar projection $\EuScript{P}_{S}^r$ on a $\mU{S}\in\maxStop{S}$ is
    \begin{equation*}\tag{\ref{eq:bound_Haar-projector}}
      \EuScript{P}_{\mU{S}} f \coloneqq \sum_{\substack{Q\in\rStop{S}\\ Q\subseteq \mU{S}}} \Delta_Q f \quad \text{ and satisfies }  \quad \langle \lvert \EuScript{P}_{\mU{S}} f\rvert \rangle_{\mU{S}} \lesssim \langle\lvert f\rvert\rangle_S.
    \end{equation*}
    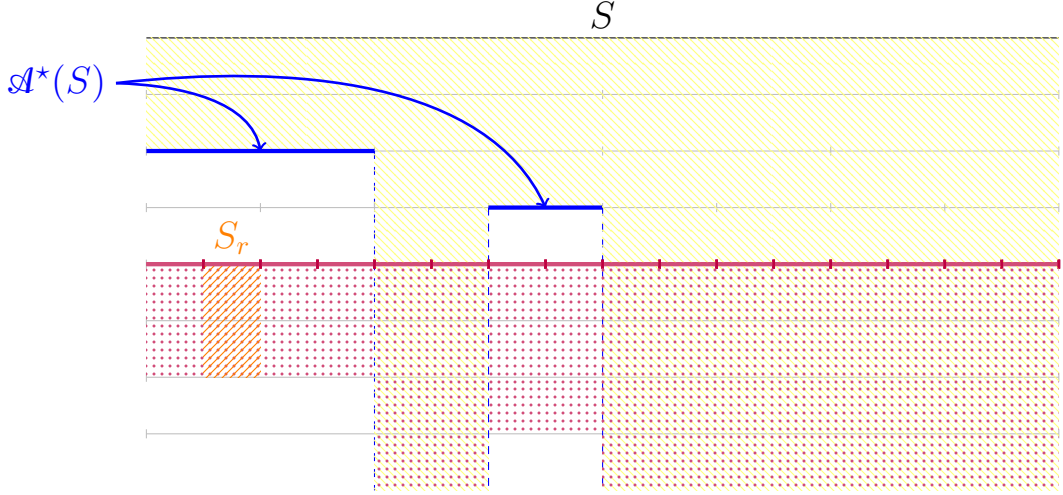
\begin{figure}\centering
      
      \begin{tikzpicture}[xscale=4,yscale=3]
  % Draw \N dyadic generations 
  \pgfmathtruncatemacro\N{3}

  % Draw other \N lines below
  \foreach \i in {0,1,2,...,\N}
  {
    \draw[lightgray,-] (0,-\i*.25 - \N*.25 - .25) -- (\N,-\i*.25 - \N*.25 -.25);
    \foreach \j in {0,\N}
    \draw[lightgray] (\j,-\i*.25 + 0.02 -\N*.25 -.25) -- (\j,-\i*.25 - 0.02 - \N*.25-.25);    
  }

  % Some labels
  \draw (\N*.5,0) node[above] {\Large $S$};
  \draw[-] (0,0) -- (\N,0);
  \draw[color=blue] (-.1,-.2) node[left] {\Large $\mathcal{A}^\star(S)$};
  
  \foreach \i in {1,2,...,\N}
  {
    \draw[lightgray,-] (0,-\i*.25) -- (\N,-\i*.25);

    \draw[lightgray] (0,-\i*.25 + 0.02) -- (0,-\i*.25 - 0.02);
    
    \pgfmathtruncatemacro\ii{2^(\i)}
    \pgfmathsetmacro\step{\N/(\ii)}

    \foreach \j in {1,...,\ii}
    \draw[lightgray] (\j*\step,-\i*.25 + 0.02) -- (\j*\step,-\i*.25 - 0.02);

    % Select stopping intervals
    \pgfmathtruncatemacro\G{\N-1}
    \ifthenelse{\i=\G}{
      \fill[pattern=north west lines, pattern color=yellow!50] (0,-\i*.25) rectangle (\step,0);

      \draw[line width=1pt,color=blue,->] (-.1,-.2) to[bend left,in=130] (\step*.5,-\i*.25); % arrow
      \draw[ultra thick,color=blue] (0,-\i*.25) -- (\step,-\i*.25);
      \draw[dashed,color=blue] (\step,-\i*.25) -- (\step,-2*\N*.25 -.5);
    }{
      \ifthenelse{\i=\N}{
        \fill[pattern=north west lines, pattern color=yellow!50] (3*\step,-\i*.25)--(4*\step,-\i*.25)--(4*\step,0)--(3*\step,0);
        \fill[pattern=north west lines, pattern color=yellow!50] (2*\step,0)--(2*\step,-2*\N*.25-.5)--(3*\step,-2*\N*.25-.5)--(3*\step,0);

        \draw[line width=1pt,color=blue,->] (-.1,-.2) to[bend left,in=130] (3.5*\step,-\i*.25); % arrow
        \draw[ultra thick,color=blue] (3*\step,-\i*.25) -- (4*\step,-\i*.25);
        \draw[dashed,color=blue] (3*\step,-\i*.25) -- (3*\step,-2*\N*.25-.5);
        \draw[dashed,color=blue] (4*\step,-\i*.25) -- (4*\step,-2*\N*.25-.5);

        % fill the rest
        \fill[pattern=north west lines, pattern color=yellow!50] (4*\step,0)--(4*\step,-2*\N*.25-.5)--(\N,-2*\N*.25-.5)--(\N,0);
      }{}
    }     

    \begin{scope}[yshift=-(\N+1)*.25 cm] % Shifted stopping tree
      \clip (-0.1,0.2) rectangle (\N+.1,-\N*.25 -.25);

      % New purple line of r-grandchildren
      \draw[line width=1.5pt,purple!70] (0,0) -- (\N,0);
      \pgfmathtruncatemacro\ss{2^(\N+1)}
      \pgfmathsetmacro\steps{\N/(\ss)}

      \foreach \j in {1,...,\ss}
      \draw[line width=1pt,purple] (\j*\steps,0.02) -- (\j*\steps,-0.02);

      \pgfmathtruncatemacro\G{\N-1}
      \ifthenelse{\i=\G}{
        % \draw[line width=1pt,color=blue] (0,-\i*.25) -- (\step,-\i*.25);
        % \draw[dashed,color=blue] (\step,-\i*.25) -- (\step,-2*\N*.25 -.5);

        \fill[pattern=dots, pattern color=purple!70] (0,-\i*.25) rectangle (\step,0);

        % one specific S_r
        \draw[color=orange] (1.5*\steps,0) node[above] {\Large $S_r$};
        \fill[pattern=north east lines, pattern color=orange] (\steps,0) rectangle (2*\steps,-\G*.25);

      }{
        \ifthenelse{\i=\N}{
          % \draw[line width=1pt,color=blue] (3*\step,-\i*.25) -- (4*\step,-\i*.25);
          % \draw[dashed,color=blue] (3*\step,-\i*.25) -- (3*\step,-2*\N*.25-.5);
          % \draw[dashed,color=blue] (4*\step,-\i*.25) -- (4*\step,-2*\N*.25-.5);

          \fill[pattern=dots, pattern color=purple!70] (3*\step,-\i*.25)--(4*\step,-\i*.25)--(4*\step,0)--(3*\step,0);
          \fill[pattern=dots, pattern color=purple!70] (2*\step,0)--(2*\step,-2*\N*.25-.5)--(3*\step,-2*\N*.25-.5)--(3*\step,0);

          % fill the rest
          \fill[pattern=dots, pattern color=purple!70] (4*\step,0)--(4*\step,-2*\N*.25-.5)--(\N,-2*\N*.25-.5)--(\N,0);
        }{}
      }     
    \end{scope}   

  }

\end{tikzpicture}

%%% Local Variables:
%%% mode: latex
%%% TeX-master: tikz.tex
%%% End:    
      \caption{An example of stopping tree $\Stop{S}$ and the maximal stopping cubes in $\mathcal{A}^\star(S)$.
        Below, shifted by $r$ generations, there is the stopping tree $\rStop{S}$.
        The cubes $Q$ in $\rStop{S}$ contained in a specific $r$-grandchild $S_r$ are highlighted.}
      \label{fig:StoppingTree}
    \end{figure}
    
    \begin{proof}[Proof of \eqref{eq:bound_Haar-projector}]      
      The Haar projector $\EuScript{P}_{\mU{S}} f$ is measurable with respect to the $\sigma$-algebra $\mathscr{G}_{S}^r$, then
      \begin{flalign*}
        \int_{\mU{S}} \lvert \EuScript{P}_{\mU{S}} f\rvert = \int_{\mU{S}} \lvert \EuScript{P}_{\mU{S}} \mathbb{E}[f\1_{\mU{S}} | \mathscr{G}_{S}^r]\rvert &
        \le \lVert \1_{\mU{S}} \rVert_{L^{p'}} \lVert \EuScript{P}_{\mU{S}} \mathbb{E}[f\1_{\mU{S}} | \mathscr{G}_{S}^r] \rVert_{L^p(\mU{S})} \nonumber \\
        \text{ by } \eqref{eq:Lp_bound_martingale} \quad & \lesssim_p \lVert \1_{\mU{S}} \rVert_{L^{p'}} \lVert \mathbb{E}[f\1_{\mU{S}} | \mathscr{G}_{S}^r] \rVert_{L^p(\mU{S})} \\
        & \le \lvert \mU{S} \rvert^{\frac{1}{p'}} \lvert \mU{S} \rvert^{\frac{1}{p}} \lVert \mathbb{E}[f\1_{\mU{S}} | \mathscr{G}_{S}^r] \rVert_\infty \\
        \text{by \cref{lemma:bound_sup_expectation} } & \lesssim \lvert \mU{S} \rvert \langle f \rangle_{S} .
      \end{flalign*}      
      Divide by $\lvert \mU{S}\rvert$ both sides to conclude.
    \end{proof}

    \section*{Acknowledgements}
    This work will be part of the author's PhD thesis supervised by Maria Carmen Reguera whose patience, support and guidance is greatly appreciated.
    The author wishes to thank Gennady Uraltsev for stimulating discussions on related topics.

    \nocite{SemmesTb,IntuitiveDyadic,HofmannSurvey}
    \printbibliography
    
  \end{document}